\documentclass[11pt,reqno]{amsart}
\usepackage[T1]{fontenc}
\usepackage{lmodern}
\usepackage[margin=1in,headheight=14pt]{geometry}
\usepackage{amsmath,amssymb,mathtools,mathrsfs}
\usepackage{booktabs,array,longtable}
\usepackage{enumitem}
\usepackage{microtype,needspace}
\numberwithin{equation}{section}
\newtheorem{theorem}{Theorem}[section]
\newtheorem{proposition}[theorem]{Proposition}
\newtheorem{lemma}[theorem]{Lemma}
\newtheorem{corollary}[theorem]{Corollary}
\theoremstyle{definition}
\newtheorem{definition}[theorem]{Definition}
\newtheorem{example}[theorem]{Example}
\theoremstyle{remark}
\newtheorem{remark}[theorem]{Remark}
\setlist[enumerate,1]{label=\textup{(\roman*)},leftmargin=2.5em,itemsep=3pt,topsep=5pt}
\allowdisplaybreaks[2]
\title[Rational equivalence closures and Lawson homology]{Topology of rational equivalence closures and Lawson homology}
\author{Wenchuan Hu}
\address{School of Mathematics, Sichuan University, Chengdu 610064, P. R. China}
\email{huwenchuan@gmail.com}
\email{wenchuan@scu.edu.cn}
\date{\today}
\begin{document}
\raggedbottom
\begin{abstract}
Rational equivalence need not be closed in the Chow topology, and its closure can have infinitely generated homotopy groups. For a connected reduced complex projective variety, we prove that this closure is locally contractible and that its quotient in the algebraically trivial cycles is a finite or countable weak product of circles. The quotient map admits a global continuous section. Consequently, the higher homotopy of the cycle space is carried by the closure and is canonically identified with Lawson homology. We describe the remaining fundamental group quotient and its integral Walker comparison. Applications include smooth projective examples with infinite rank closure homotopy throughout prescribed finite intervals and a family of abelian threefolds with four simultaneously dense, nonconstructible quotient rank strata.
\end{abstract}
\maketitle

\section{Introduction}
Rational equivalence is the relation defining Chow groups and the intersection theory of algebraic cycles. Its behavior in the topology of cycle spaces raises a different question: which cycles can be approximated by rationally trivial cycles, and what topology is carried by their closure? The passage to the closure need not preserve rational equivalence, and neither the resulting subgroup nor its homotopy groups are determined by the abstract definition of a Chow group. Understanding this passage is therefore part of understanding the topology of algebraic cycles, rather than merely a separation convention for their quotient.

A precise motivation appears in Lima-Filho's discussion of rational equivalence closures \cite[Section 3.5.3, pp.~104--106]{LF04}. In the \emph{Speculations and remarks} on pp.~105--106, he proposes a filtration with smooth successive quotients and suggests that the more complicated homotopy of the cycle group should be concentrated in the closure of rational equivalence. This is a proposed picture, not a theorem about the topology of that closure. It leads to concrete questions about local contractibility, higher homotopy, and the actual Hausdorff quotient. These questions cannot be answered simply by observing that the subgroup is closed: closed subgroups of locally contractible groups need not themselves be locally contractible.

We prove a structure theorem in the actual Chow topology that makes this concentration of higher homotopy precise. The quotient is a weak product of circles, whereas the closure retains all homotopy groups of the cycle space in degrees at least two. Local contractibility does not imply finite generation of the homotopy groups of the closure. Indeed, the examples constructed below have infinite rank homotopy in any prescribed finite interval of degrees at least two. At the same time, classical examples have finitely generated homotopy groups. Thus the assertion is not that every closure has infinite rank homotopy, but that its homotopy groups can be infinitely generated even though the quotient is aspherical.

\medskip
\noindent\textit{Conventions.}
All varieties are reduced and defined over $\mathbb{C}$. Unless stated otherwise, $X$ is smooth, connected, and projective of dimension $n\geq1$; each structural result that permits singular, reducible, or nonequidimensional varieties is stated in that generality. We use ordinary dimension indexed Chow groups on singular varieties, including cycles supported in the singular locus. The integer $p$ denotes cycle dimension, $m$ denotes homotopy degree, and $k=2p+m$ denotes Lawson homological degree. For smooth $X$, write $r=n-p$ for codimension.

All homotopy groups are based at the zero cycle, and singular homology has integral coefficients unless another coefficient group is displayed. For an abelian group $B$, set $B_{\mathbb{Q}}=B\otimes_{\mathbb{Z}}\mathbb{Q}$, let $B_{\mathrm{tors}}$ be its torsion subgroup, and write $B_{\mathrm{free}}=B/B_{\mathrm{tors}}$ and $\operatorname{rank}B=\dim_{\mathbb{Q}}B_{\mathbb{Q}}$. We set $H_j(-,\mathbb{Z})=0$ for $j<0$. Write $\mathbb{T}=\mathbb{R}/\mathbb{Z}$, with identity $0$. For a topological group, $\pi_0=0$ means path connectedness. We denote homotopy equivalence by $\simeq$ and use $K(B,j)$ for an Eilenberg--Mac Lane space. All asserted product decompositions are specified as group isomorphisms, homeomorphisms, or weak homotopy equivalences; these assertions are not interchangeable.

For smooth projective $X$, put $b_j(X)=\dim_{\mathbb{Q}}H_j(X,\mathbb{Q})$, $h^{a,b}(X)=\dim_{\mathbb{C}}H^{a,b}(X)$, $q(X)=h^{1,0}(X)$, and $p_g(X)=h^{n,0}(X)$. The notation $\operatorname{PD}$ denotes integral Poincar\'e duality, and
\[
 F^aH^j(X,\mathbb{C})=\bigoplus_{u\geq a}H^{u,j-u}(X)
\]
is the Hodge filtration. For groups that need not be finitely generated, we use filtered colimits of integral mixed Hodge structures, rather than finite type mixed Hodge structures.

Let $Z_p(X)$ be the topological group of algebraic $p$-cycles with the Chow topology, and put
\begin{equation}\label{eq:groups}
 G_p(X)=Z_p(X)_{\mathrm{alg}},\qquad R_p(X)=Z_p(X)_{\mathrm{rat}},\qquad
 \operatorname{CR}_p(X)=\overline{R_p(X)}\subset Z_p(X).
\end{equation}
The subscripts $\mathrm{alg}$ and $\mathrm{rat}$ mean algebraic and rational equivalence to zero, respectively. We use the dimension indexed Chow group $\operatorname{CH}_p(X)=Z_p(X)/R_p(X)$; for smooth $X$, $\operatorname{CH}^{r}(X)=\operatorname{CH}_{n-r}(X)$. The bar always means closure in the actual Chow topology. The closure carries its induced topology, not a topology defined by parameter spaces of rational equivalences. We distinguish cycles from their Chow classes by writing
\begin{equation}\label{eq:abstract}
 A_p(X)=\operatorname{CH}_p(X)_{\mathrm{alg}}=G_p(X)/R_p(X),\qquad
 D_p(X)=\operatorname{CR}_p(X)/R_p(X)\subset A_p(X).
\end{equation}
Unless a quotient topology is explicitly mentioned, $A_p(X)$ and $D_p(X)$ are regarded as abstract groups. For a connected projective parameter variety $T$, the notation $Z_0(T)^0$ means its degree zero cycle subgroup. In contrast,
\begin{equation}\label{eq:quotient}
 Q_p(X)=G_p(X)/\operatorname{CR}_p(X),\qquad q_p:G_p(X)\longrightarrow Q_p(X)
\end{equation}
is a Hausdorff topological group. Abstractly, $Q_p(X)=A_p(X)/D_p(X)$. It is the maximal Hausdorff quotient of the algebraically trivial part of the topological Chow group. Lawson homology is
\begin{equation}\label{eq:lawson}
 L_pH_{2p+m}(X)=\pi_mZ_p(X)\quad(m\geq0).
\end{equation}
For $m>0$, the cycle space in this formula may be replaced by $G_p(X)$. Let $\Phi_{p,k}:L_pH_k(X)\to H_k(X,\mathbb{Z})$ be the natural cycle map, and put $L_pH_k(X)_{\mathrm{hom}}=\ker\Phi_{p,k}$. Likewise, $Z_p(X)_{\mathrm{hom}}$ denotes the subgroup of cycles with zero integral homology class.

For smooth projective $X$, continuity of the classical Abel--Jacobi map on homologically trivial cycles \cite[Proposition 3.5.12(iii), p.~105]{LF04}, restricted to $G_p(X)$, gives
\[
 R_p(X)\subset\operatorname{CR}_p(X)
 \subset\ker\bigl(\operatorname{AJ}|_{G_p(X)}\bigr).
\]
Lima-Filho's codimension indexed notation $Z^a(X)$ corresponds to $Z_{n-a}(X)$ here. The second inclusion can be strict, as explained in Section~\ref{sec:invariants}. We do not construct the entire filtration proposed in \cite{LF04}. Unless specified otherwise, every classical Abel--Jacobi kernel in this paper is taken on $G_p(X)$, not on all homologically trivial cycles.
Define
\begin{equation}\label{eq:lambda}
 \Lambda_p(X)=\pi_1Q_p(X),\qquad 0\leq p\leq n.
\end{equation}
We call these groups quotient lattices; the term permits countably infinite rank. They are defined through the Hausdorff Chow quotients. At the classical indices they recover familiar integral lattices; the general comparison involves an additional free abelian group. The top index is included: $G_n=R_n=\operatorname{CR}_n=Q_n=\Lambda_n=0$.

\begin{theorem}\label{thm:main}
Let $X$ be a connected reduced complex projective variety of dimension $n\geq1$, and let $0\leq p<n$.
\begin{enumerate}
\item There are compact connected real tori $\mathcal{A}_d\subset \mathcal{A}_{d+1}\subset Q_p(X)$, with closed subgroup inclusions, such that
\[
 Q_p(X)=\bigcup_{d\geq1}\mathcal{A}_d
\]
with the final topology of this exhaustion. Noncanonically, as topological groups,
\begin{equation}\label{eq:weakproduct}
 Q_p(X)\cong\mathop{\bigoplus}\nolimits_{j\in I}\mathbb{T},
\end{equation}
where $I$ is finite or countable and the right-hand side has finite support and the final topology of its finite coordinate tori. In particular, $\pi_mQ_p(X)=0$ for $m\geq2$.
\item The quotient map $q_p$ is a numerable principal $\operatorname{CR}_p(X)$-bundle and a Hurewicz fibration. The group $\operatorname{CR}_p(X)$ is path connected, locally contractible, and of CW homotopy type. Moreover, $q_p$ admits a global continuous section through zero, yielding a noncanonical homeomorphism of spaces
\begin{equation}\label{eq:globalproduct}
 G_p(X)\cong\operatorname{CR}_p(X)\times Q_p(X)
\end{equation}
over $Q_p(X)$. Unless $Q_p(X)=0$, no section can be a group homomorphism.
\item The group
\[
 \Lambda_p(X)=\operatorname*{colim}_d H_1(\mathcal{A}_d,\mathbb{Z})
\]
is countable free abelian. There is a canonical exact sequence
\begin{equation}\label{eq:exact}
 0\longrightarrow\pi_1\operatorname{CR}_p(X)\longrightarrow L_pH_{2p+1}(X)
 \xrightarrow{\ \kappa_p\ }\Lambda_p(X)\longrightarrow0,
 \qquad\kappa_p=(q_p)_*,
\end{equation}
and inclusion induces canonical isomorphisms
\begin{equation}\label{eq:higher}
 \pi_m\operatorname{CR}_p(X)\xrightarrow{\ \sim\ } L_pH_{2p+m}(X)\qquad(m\geq2).
\end{equation}
\end{enumerate}
\end{theorem}

The weak product in \eqref{eq:weakproduct} is an aspherical CW complex. Consequently $Q_p(X)$ has homotopy type $K(\Lambda_p(X),1)$, and its universal cover is contractible by Whitehead's theorem. These conclusions have no dimension-dependent upper bound.

\begin{corollary}\label{cor:moore}
Noncanonically, the weak homotopy type of the closure is
\begin{equation}\label{eq:moore}
 K(\ker\kappa_p,1)\times\prod_{m\geq2}K\bigl(L_pH_{2p+m}(X),m\bigr).
\end{equation}
Furthermore, $\bigl(\pi_1\operatorname{CR}_p(X)\bigr)_{\mathrm{tors}}\cong\bigl(L_pH_{2p+1}(X)\bigr)_{\mathrm{tors}}$.
\end{corollary}
The proof is given in Section~\ref{sec:sections}. Formula~\eqref{eq:higher} identifies where higher homotopy occurs; it does not compute the Lawson groups for a general $X$. In particular, replacing Lawson homology by singular homology requires a separate comparison theorem. Nor does density of $R_p(X)$ identify that subgroup with its closure up to weak equivalence.

The topology of the closure has the following consequence for infinite generation. The Lawson existence assertion was stated in \cite[Theorem 1.7 and Corollary 1.8]{HuAJ}. Section~\ref{sec:infinite} supplies the passage from Clemens's theorem to the required projected invariant and proves the following explicit form together with its consequences for closures.

\begin{theorem}\label{thm:infiniteexistence}
There is a smooth quintic threefold $Y\subset\mathbb{P}^4$ such that, for every elliptic curve $E$ and all integers $p\geq1$, $m\geq0$, the smooth projective variety
\begin{equation}\label{eq:infinitevariety}
 X_{p,m}=E^m\times Y\times\mathbb{P}^{p-1},\qquad \dim X_{p,m}=p+m+2,
\end{equation}
satisfies
\begin{equation}\label{eq:infinitelawson}
 \dim_{\mathbb{Q}}\bigl(L_pH_{2p+m}(X_{p,m})_{\mathrm{hom}}\otimes\mathbb{Q}\bigr)=\infty.
\end{equation}
Consequently, for every $k\geq2p>1$ there is a smooth projective $X$ of dimension $k-p+2$ for which $L_pH_k(X)$ is not finitely generated. For $m\geq2$, one also has
\begin{equation}\label{eq:infiniteclosure}
 \dim_{\mathbb{Q}}\bigl(\pi_m\operatorname{CR}_p(X_{p,m})\otimes\mathbb{Q}\bigr)
 =\dim_{\mathbb{Q}}H_m(\operatorname{CR}_p(X_{p,m}),\mathbb{Q})=\infty.
\end{equation}
\end{theorem}
The quintic may be chosen in one dense $G_\delta$ subset of the analytic parameter space of smooth quintics. This analytic genericity is sufficient for existence; no algebraicity of the additional exceptional sets is assumed. Corollary~\ref{cor:simultaneous} realizes a prescribed finite interval of Lawson degrees on one fixed smooth projective variety. In degree $m=1$ the exact sequence \eqref{eq:exact}, rather than \eqref{eq:higher}, governs the conclusion.

The results allow singular, reducible and nonequidimensional projective varieties. Smoothness is imposed only for the classical Jacobians and pure Hodge theory. For singular targets all Chow groups are ordinary dimension indexed Chow groups, including cycles in the singular locus.

The second structural comparison is the integral quotient
\[
 \Lambda_p(X)\twoheadrightarrow H_1(J_W^{2r-1}(X),\mathbb{Z}),
 \qquad r=\dim X-p,
\]
proved in Theorem~\ref{thm:walkerquotient}. Here $J_W^{2r-1}(X)$ is the Walker intermediate Jacobian, defined in Section~\ref{sec:invariants}. We retain its integral lattice, rather than replacing it by an isogenous torus. The remaining factor measures precisely the part of the quotient lattice not detected by the usual cycle map. Proposition~\ref{prop:comparisoncriterion} shows that this factor vanishes when the rational Lawson cycle map in degree $2p+1$ is injective.

Sections~\ref{sec:topology}--\ref{sec:sections} prove the closure theorem. Sections~\ref{sec:torsion} and~\ref{sec:invariants} treat regular quotients, torsion and integral periods. The classical calculations in Section~\ref{sec:examples} illustrate the hypotheses. The coniveau computation in Section~\ref{sec:variation} uses established Hodge group results; its additional point is the simultaneous density of four rank strata. Finally, Section~\ref{sec:infinite} supplies the projected period argument required for the infinite rank applications. The Lawson existence statement was already made in \cite{HuAJ}; we include the missing passage from the full Abel--Jacobi invariant to its top form projection rather than treating the older statement as a new result.

\section{The topology of cycle groups}\label{sec:topology}
Let $X$ be any connected reduced complex projective variety. Fix an embedding and let $C_{p,d}(X)$ be its projective Chow variety of effective $p$-cycles of degree $d$. Write $M_p(X)=\coprod_{d\geq0}C_{p,d}(X)$. Its group completion is $Z_p(X)$, with compact filtration
\begin{equation}\label{eq:Kd}
 K_d=\operatorname{im}\left(\coprod_{a+b\leq d}C_{p,a}(X)\times C_{p,b}(X)
 \xrightarrow{(u,v)\mapsto u-v} Z_p(X)\right).
\end{equation}
The final topology of the compact Hausdorff spaces $K_d$ is Hausdorff and independent of the embedding \cite{Fri91,LF92}; see also \cite[Section 3.2, Properties 1 and 5, pp.~78--79; Section 3.2.2, equation (4), p.~82]{LF04}. Thus $Z_p(X)$ is a $k_\omega$ topological abelian group: its topology is determined by a countable increasing compact exhaustion. Finite products of Hausdorff $k_\omega$ spaces are $k_\omega$, so the finite product topologies used here agree with their compactly generated versions. Addition restricts to continuous maps $K_d\times K_e\longrightarrow K_{d+e}$. Theorem~3.2.16 of \cite[p.~83]{LF04} identifies this topology with the topology generated by flat algebraic families in Example~3.2.4; Corollary~3.2.9 identifies the latter with the equidimensional-family topology. Here ``flat'' refers to the algebraic families, not to an unfiltered flat-norm topology on currents.

Property~1 in \cite[p.~78]{LF04} asserts CW homotopy type. Local contractibility requires a statement about the actual space. We use \cite[Theorem 1.5(d), with Definition 1.2 and Lemma 1.3]{Fri98} with source a point and ambient target $\mathbb{P}^N$. The monoid $E_p=M_p(X)\subset M_p(\mathbb{P}^N)$ is cancellative. In every degree, the condition that a cycle be supported on the closed subvariety $X$ is Zariski closed. Thus $E_p$ is constructibly embedded in the sense of that theorem, and the zero submonoid $F_p=0$ is Zariski closed. No normality or equidimensionality of $X$ is needed, since the normal ambient target is $\mathbb{P}^N$.

The naive group completion used in \cite{Fri98} has the quotient topology from $E_p\times E_p$. Its restriction to a finite union of degree pairs is a continuous surjection from a compact space to the Hausdorff set $K_d$, hence a quotient map. Consequently a subset of the completion is open if and only if its inverse image on every degree pair is open, equivalently if and only if its intersection with every $K_d$ is open. This is precisely the topology of \eqref{eq:Kd}. Theorem~1.5(d) of \cite{Fri98} therefore supplies a CW structure on $Z_p(X)$ with the topology used here. In particular, $Z_p(X)$ is locally contractible. We use the neighborhood formulation: every neighborhood of a point contains a smaller neighborhood whose inclusion in the original neighborhood is nullhomotopic. The Hurewicz fibration used below is established by our local section argument, not inferred from a weaker use of the word \emph{fibration} in the literature.

\begin{lemma}\label{lem:components}
The subgroup $G_p(X)$ is open and closed in $Z_p(X)$ and is its path component of zero. The subgroup $R_p(X)$ is path connected. Its closure is connected and contained in $G_p(X)$.
\end{lemma}
\begin{proof}
The standard component calculation identifies $\pi_0Z_p(X)$ with cycles modulo algebraic equivalence \cite{Fri91,LF92,Walker}. On each compact presentation in \eqref{eq:Kd}, the class is constant on each of finitely many connected components. The final topology therefore makes the class map to the discrete class group continuous, proving openness and closedness. The curve family description gives the path component assertion. A generator of rational equivalence is a difference of fibers over $\mathbb{P}^1$. A parameter path, followed by subtraction of its initial fiber, joins this difference to zero inside $R_p(X)$. Finite sums and inverses prove path connectedness. Taking closure gives the final assertion.
\end{proof}

We shall repeatedly use the curve presentation
\begin{equation}\label{eq:curvepresentation}
 z=\sum_{i=1}^s\epsilon_i\Gamma_{f_i}(u_i),\qquad
 \Gamma_f\left(\sum_jn_j[t_j]\right)=\sum_jn_jf(t_j),
\end{equation}
where $z\in G_p(X)$, $\epsilon_i\in\{1,-1\}$, $D_i$ is a smooth connected projective curve, $f_i:D_i\to C_{p,d_i}(X)$ is a morphism, and $u_i\in Z_0(D_i)^0$. This is an equality of cycles. The curve description of algebraic equivalence gives such a presentation modulo rational equivalence \cite[Section 10.3]{Fulton}; any rationally trivial residual cycle is a sum of $\mathbb{P}^1$-family differences \cite[proof of Theorem 5.7(2)]{Walker}. Properness of the Chow variety extends morphisms from smooth parameter curves to their projective completions.

We also use the following precise fiber theorem: a Hurewicz fibration whose total space and base have CW homotopy type has fibers of CW homotopy type \cite[Proposition 3]{Schon}.

\section{Zero-cycles on parameter spaces}\label{sec:parameters}
The parameter spaces in this section are smooth and projective, whereas the target of a cycle family may be singular. For such a parameter space $T$, write $\operatorname{Alb}(T)$ for its Albanese variety. For a smooth projective curve $D$, put $J(D)=\operatorname{Pic}^0(D)$ and let $g(D)$ denote its genus; $\operatorname{Sym}^eD$ is its $e$th symmetric power.

\begin{lemma}\label{lem:albsections}
For a smooth connected projective variety $T$, the continuous Albanese homomorphism
\[
 a_T:Z_0(T)^0\longrightarrow\operatorname{Alb}(T)
\]
is surjective and admits continuous local sections through any prescribed point of its total space.
\end{lemma}
\begin{proof}
The assertion is immediate for a point. Otherwise choose a smooth ample complete intersection curve $D\subset T$ for which $J(D)\to\operatorname{Alb}(T)$ is surjective; take $D=T$ in dimension one. A surjective homomorphism of compact Lie groups has a local section at zero. For $e\geq\max\{0,2g(D)-1\}$, the Abel map $\operatorname{Sym}^eD\to\operatorname{Pic}^e(D)$ is a projective bundle. Choose an effective divisor $E$ of degree $e$, identify $J(D)$ with $\operatorname{Pic}^e(D)$ by adding $E$, and take a local section through $E$. Composing the two sections, subtracting $E$, and pushing to $T$ gives a local section through the zero cycle. Translations give sections through any prescribed cycle. Surjectivity follows from the same Jacobian and Abel maps.
\end{proof}

\begin{proposition}\label{prop:albkernel}
Put $K_0(T)=\ker a_T$ with its induced topology. Without using Roitman's theorem, this space is path connected, locally contractible, and of CW homotopy type, and
\begin{equation}\label{eq:albkernel}
 \pi_1K_0(T)=H_1(T,\mathbb{Z})_{\mathrm{tors}},\qquad
 \pi_mK_0(T)=H_m(T,\mathbb{Z})\quad(m\geq2).
\end{equation}
\end{proposition}
\begin{proof}
Lemma~\ref{lem:albsections} makes $a_T$ a locally trivial principal $K_0(T)$-bundle over a compact torus, hence a numerable bundle and a Hurewicz fibration \cite{Dold}. A local section $s$ through zero gives the retraction $z\mapsto z-s(a_T(z))$ from an open subset of $Z_0(T)^0$ to $K_0(T)$. Local contractibility follows from the actual CW structure of the cycle space and the neighborhood definition of local contractibility.

Dold--Thom gives $\pi_mZ_0(T)^0=H_m(T,\mathbb{Z})$ for $m\geq1$ \cite[Section 2]{LF92}. The period definition of the Albanese map identifies its map on $\pi_1$ with $H_1(T,\mathbb{Z})\twoheadrightarrow H_1(T,\mathbb{Z})_{\mathrm{free}}=H_1(\operatorname{Alb}(T),\mathbb{Z})$. This is independent of a Chow-group torsion theorem. The homotopy exact sequence and asphericity of the base give path connectedness and \eqref{eq:albkernel}. The total component is a CW subcomplex, so the fiber theorem quoted in Section~\ref{sec:topology} gives CW homotopy type.
\end{proof}

\begin{proposition}\label{prop:zero-closure}
For every smooth connected projective $T$,
\begin{equation}\label{eq:zero-closure}
 \overline{R_0(T)}=\ker a_T.
\end{equation}
\end{proposition}
\begin{proof}
Continuity and rational invariance give the inclusion from left to right. Write an element of the right-hand side as $z=u-v$, with effective cycles $u,v$ of equal degree. In dimension at least two, choose a smooth projective ample curve $D\subset T$ through their finite supports; in dimension one take $D=T$. The zero-dimensional case is trivial. The divisor class $b=[u-v]\in J(D)$ belongs to the algebraic subgroup $B=\ker(J(D)\to\operatorname{Alb}(T))$, which need not be connected.

Torsion points are dense in every component of $B$. Indeed, torsion is dense in $B^0$, and every component has a torsion representative: if $Nx\in B^0$, choose $y\in B^0$ with $Ny=Nx$, so $x-y$ is such a representative. Choose torsion points $b_\nu\to b$. Add a fixed effective divisor $E$ so that $u+E$ and $v+E$ have common degree at least $2g(D)-1$. A local section of the Abel projective bundle through $u+E$ produces effective divisors $u_\nu\to u+E$ satisfying $[u_\nu-(v+E)]=b_\nu$. Pushing to $T$ gives $z_\nu\to z$ in a fixed bounded degree stage. The classes $[z_\nu]$ are torsion and Albanese-trivial. Roitman's torsion injectivity \cite{Roitman}, also proved in \cite[Theorem 4.2]{Bloch}, gives $[z_\nu]=0$. Thus $z$ lies in the closure of $R_0(T)$.
\end{proof}

\begin{remark}\label{rem:roitman}
Proposition~\ref{prop:albkernel} is topological and does not prove Roitman's theorem. Proposition~\ref{prop:zero-closure} uses that theorem independently. Algebraically, torsion injectivity says that $R_0(T)$ is saturated in $K_0(T)$:
\[
 z\in K_0(T),\quad Nz\in R_0(T),\quad N\geq1\quad\Longrightarrow\quad z\in R_0(T).
\]
Although $K_0(T)/R_0(T)$ has the indiscrete quotient topology by density, that topology does not determine its abstract torsion. This distinction prevents a circular deduction of a torsion theorem from local contractibility.
\end{remark}

\begin{lemma}\label{lem:familyfactor}
Let $T$ be smooth connected projective and $f:T\to C_{p,d}(X)$ a morphism. The continuous homomorphism $\Gamma_f:Z_0(T)^0\to G_p(X)$ sends $R_0(T)$ into $R_p(X)$. There is a unique continuous homomorphism
\begin{equation}\label{eq:familyfactor}
 \alpha_f:\operatorname{Alb}(T)\longrightarrow Q_p(X),\qquad q_p\Gamma_f=\alpha_fa_T.
\end{equation}
\end{lemma}
\begin{proof}
On symmetric powers, $\Gamma_f$ is addition of the cycles $f(t)$ and multiplies the degree bound by $d$, proving continuity on group completions. The image of a principal divisor on a parameter curve is rationally trivial, by the $\mathbb{P}^1$-family description of rational equivalence \cite[proof of Theorem 5.7(2)]{Walker}; this holds for singular targets. Continuity and Proposition~\ref{prop:zero-closure} imply that $q_p\Gamma_f$ kills $\ker a_T$. It therefore factors uniquely, and the factor is continuous because Lemma~\ref{lem:albsections} makes $a_T$ a quotient map.
\end{proof}
Resolving entire projective Chow components, rather than using unrelated individual curves, will now provide finitely many parameter spaces for each bounded degree stage.

\section{The Hausdorff quotient is a union of tori}\label{sec:tori}
Here and in Section~\ref{sec:sections}, $X$ is allowed to be singular and reducible.

\begin{proposition}\label{prop:exhaustion}
There are compact connected real tori $\mathcal{A}_1\subset\mathcal{A}_2\subset\cdots\subset Q_p(X)$ such that
\[
 q_p(K_d\cap G_p(X))\subset \mathcal{A}_d,\qquad Q_p(X)=\bigcup_d\mathcal{A}_d.
\]
The final topology of this exhaustion is the quotient topology. Moreover, for each $d$ there is an integer $e(d)$ with $\mathcal{A}_d\subset q_p(K_{e(d)}\cap G_p(X))$.
\end{proposition}
\begin{proof}
The quotient is Hausdorff because $\operatorname{CR}_p(X)$ is a closed subgroup of the Hausdorff topological group $G_p(X)$. For a fixed $d$, only finitely many irreducible Chow components occur in \eqref{eq:Kd}. Resolve each one by a smooth connected projective variety $T_V\to V$ and choose a base point $t_V$. For the corresponding family $f_V$, Lemma~\ref{lem:familyfactor} places
\[
 q_p\bigl(f_V(t)-f_V(t_V)\bigr)
 \quad\text{in}\quad \operatorname{im}\bigl(\operatorname{Alb}(T_V)\to Q_p(X)\bigr).
\]
For a pair of components $V,W$ whose differences lie in $G_p(X)$, the quotient of any difference $a-b$ is a sum of two such Albanese images and the fixed element $q_p(f_V(t_V)-f_W(t_W))$. The cycle inside this last expression is algebraically trivial. The literal curve presentation \eqref{eq:curvepresentation} and Lemma~\ref{lem:familyfactor} put its image in a finite sum of curve Jacobian images. Signs are absorbed by inversion on the parameter Jacobians.

Thus $q_p(K_d\cap G_p(X))$ lies in the image of a finite product $B_d$ of complex abelian varieties. Include all previously chosen factors and let $\mathcal{A}_d$ be this image. A continuous homomorphism from a compact torus to a Hausdorff group identifies its image with the quotient by a closed subgroup. Such a quotient is a compact abelian Lie group, and here it is connected. Hence $\mathcal{A}_d$ is a real torus. The inclusions are closed, and their union is $Q_p(X)$.

For the bounded presentation, cover each nontrivial parameter Albanese by the Jacobian of a smooth curve $D_i$. For a sufficiently large effective divisor $E_i$ on $D_i$, the Abel map from $\operatorname{Sym}^{\deg E_i}D_i$, with $E_i$ subtracted, covers that Jacobian. The product of these symmetric powers maps onto $B_d$. Applying the finitely many cycle-family maps yields a compact set of cycles of one bounded degree mapping onto $\mathcal{A}_d$. This proves the assertion involving $e(d)$.

If $U\subset Q_p(X)$ is open, each $U\cap \mathcal{A}_d$ is open. Conversely, suppose all these intersections are open. The continuous maps $K_d\cap G_p(X)\to \mathcal{A}_d$ give open inverse images of $U$. Since $G_p(X)$ is open and closed, its topology is final for $K_d\cap G_p(X)$. Hence $q_p^{-1}(U)$ is open, and the quotient topology makes $U$ open.
\end{proof}

\begin{lemma}\label{lem:tori}
A countable direct limit of compact connected real tori with closed subgroup inclusions is noncanonically a finite or countable weak product of circles. It is a countable CW complex, is paracompact, and satisfies
\[
 \pi_1Q=\operatorname*{colim}_dH_1(\mathcal{A}_d,\mathbb{Z}),\qquad \pi_mQ=0\quad(m\geq2).
\]
Its fundamental group is free abelian.
\end{lemma}
\begin{proof}
For $\mathcal{A}_d\subset \mathcal{A}_{d+1}$ the quotient is a torus. The injection on fundamental group lattices is primitive, since its cokernel is the free lattice of $\mathcal{A}_{d+1}/\mathcal{A}_d$. Choose a basis for $H_1(\mathcal{A}_1,\mathbb{Z})$ and extend it successively. A homomorphism of real tori is determined by its lattice map: its lift to universal covers is real linear. These bases therefore give compatible coordinates identifying $\mathcal{A}_d$ with the first $\operatorname{rank} H_1(\mathcal{A}_d,\mathbb{Z})$ circle factors. The final topology gives the asserted weak product.

Give each circle its finite CW structure with $0$ as a vertex. The coordinate tori are finite subcomplexes; their union has a countable CW structure and is paracompact. A compact subset of a CW complex lies in a finite subcomplex. Consequently sphere maps and their homotopies lie in finite coordinate tori, giving the homotopy calculation. The compatible lattice bases identify the fundamental group with a direct sum of copies of $\mathbb{Z}$.
\end{proof}

The use of real tori is intentional: the kernel of a parameter Albanese map to $Q_p(X)$ is known here to be a closed topological subgroup, not necessarily a complex abelian subvariety. Primitivity concerns the inclusions of the actual image tori $\mathcal{A}_d$, not the surjections $B_d\to \mathcal{A}_d$, which may have finite kernels.

\Needspace{9\baselineskip}
\section{Sections and the homotopy calculation}\label{sec:sections}
\subsection{Local sections and the fiber}
\begin{proposition}\label{prop:localsections}
The quotient map $q_p:G_p(X)\to Q_p(X)$ admits a continuous local section through zero.
\end{proposition}
\begin{proof}
For a fixed image torus $\mathcal{A}_d$, its construction gives a surjection
\[
 b_d:B_d=\prod_{i=1}^s\operatorname{Alb}(T_i)\longrightarrow \mathcal{A}_d
\]
from parameter Albaneses. Choose local sections $\ell_i:V_i\to Z_0(T_i)^0$ of their Albanese maps, with $\ell_i(0)=0$. The finite sum
\[
 \ell(t_1,\ldots,t_s)=\sum_i\Gamma_{f_i}\ell_i(t_i)
\]
is continuous on $V=\prod_iV_i$ and satisfies $q_p\ell=b_d|_V$. The Lie-group surjection $b_d$ has a local section $j:W\to B_d$ through zero, even when its kernel is disconnected. Shrinking $W$ to ensure $j(W)\subset V$ gives a local lift $\ell j$ of $\mathcal{A}_d$.

Choose compatible circle coordinates from Lemma~\ref{lem:tori}. Each inclusion $\iota_j:\mathbb{T}\hookrightarrow Q_p(X)$ lies in some $\mathcal{A}_d$, so it admits a lift $s_j:I_j\to G_p(X)$ on an open arc $I_j$ through zero, with $s_j(0)=0$. Let $U$ consist of the finite-support points whose $j$th coordinate lies in $I_j$ for every $j$. Its intersection with every finite coordinate torus is open, so $U$ is open. Define
\begin{equation}\label{eq:localsection}
 s((t_j))=\sum_js_j(t_j),\qquad (t_j)\in U.
\end{equation}
The sum is finite at each point. It is continuous on every $U\cap \mathcal{A}_d$. For a subset $V\subset U$, the condition that $V\cap\mathcal{A}_d$ be open in $U\cap\mathcal{A}_d$ for every $d$ makes these intersections open in $\mathcal{A}_d$, since $U\cap\mathcal{A}_d$ is open. The final topology then makes $V$ open in $Q_p(X)$. Thus the induced topology on $U$ is indeed final for its intersections with the finite tori. Thus $s$ is continuous, $q_ps=\operatorname{id}_U$, and $s(0)=0$. No convergence of an infinite series is involved.
\end{proof}

\begin{proof}[Proof of Theorem~\ref{thm:main}]
Part (i) follows from Proposition~\ref{prop:exhaustion} and Lemma~\ref{lem:tori}. Translate the section of Proposition~\ref{prop:localsections} to obtain local sections everywhere. A section on $U$ gives the homeomorphism
\[
 \operatorname{CR}_p(X)\times U\longrightarrow q_p^{-1}(U),\qquad(c,u)\longmapsto c+s(u),
\]
with inverse $g\mapsto(g-s(q_p(g)),q_p(g))$. Therefore $q_p$ is a locally trivial principal bundle. Its CW base is paracompact, so the bundle is numerable and is a Hurewicz fibration \cite{Dold}.

The map $g\mapsto g-s(q_p(g))$ retracts the open subset $q_p^{-1}(U)$ of $G_p(X)$ onto $\operatorname{CR}_p(X)$. A retract of a locally contractible space is locally contractible in the neighborhood sense: shrink an ambient neighborhood so that its contraction stays in the inverse image of a prescribed neighborhood under the retraction, and then compose with the retraction. This applies to the actual cycle topology by Section~\ref{sec:topology}. The same neighborhood property implies local path connectedness, since the smaller neighborhood lies in a single path component of the larger open set. The closure is connected by Lemma~\ref{lem:components}, hence path connected. The total component is a CW subcomplex and the base is a CW complex, so \cite[Proposition 3]{Schon} gives CW homotopy type of the fiber.

The homotopy exact sequence now yields \eqref{eq:higher} and
\[
 0\to\pi_1\operatorname{CR}_p(X)\to\pi_1G_p(X)\to\pi_1Q_p(X)\to0.
\]
This proves (iii). Finally, Lemma~\ref{lem:global-section} below applies to this fibration with path-connected kernel and proves the global section and \eqref{eq:globalproduct}. Its last assertion also proves the claimed failure of a group section when the quotient is nonzero. The global argument therefore uses, rather than assumes, the connectedness of the fiber established above.
\end{proof}

\begin{proof}[Proof of Corollary~\ref{cor:moore}]
Moore's splitting theorem for topological abelian groups, recalled in \cite[p.~295, before Theorem 11.5]{McCord}, identifies their weak homotopy type with a product of Eilenberg--Mac Lane spaces. Apply it to $\operatorname{CR}_p(X)$ and use Theorem~\ref{thm:main}. The torsion assertion follows from \eqref{eq:exact}, whose quotient is torsion free.
\end{proof}

\begin{remark}\label{rem:density}
Density of $R_p(X)$ in $\operatorname{CR}_p(X)$ does not imply that their inclusion is a weak equivalence. All local and homotopy statements here concern the closure with its induced topology, not the unclosed rational equivalence subgroup.
\end{remark}

\subsection{Functoriality and rank bounds}
\begin{proposition}\label{prop:functoriality}
For projective varieties, without smoothness assumptions, proper pushforward, flat pullback of relative dimension $e$, and exterior product with a fixed $q$-cycle give continuous maps
\[
 \begin{split}
 f_* &: \operatorname{CR}_p(X)\longrightarrow\operatorname{CR}_p(Y),\qquad
 f^* : \operatorname{CR}_p(Y)\longrightarrow\operatorname{CR}_{p+e}(X),\\
 (-)\times\zeta &: \operatorname{CR}_p(X)\longrightarrow\operatorname{CR}_{p+q}(X\times Y).
 \end{split}
\]
The groups and maps are independent of the projective embedding. These operations induce maps on $Q_p$ and $\Lambda_p$, and \eqref{eq:exact} and \eqref{eq:higher} are natural for them.
\end{proposition}
\begin{proof}
The Chow topology is independent of the embedding, and the identity on cycles preserves both equivalence subgroups and their closures. Proper pushforward and flat pullback are continuous by \cite[Section 3.2, Property 3, p.~78; Corollary 3.2.10, p.~81]{LF04}. Continuity of exterior product is \cite[Claim 3.3.7, p.~89; Section 3.6, item 1, p.~106]{LF04}; restricting one factor gives the operation used here. These operations preserve algebraic and rational equivalence. If $h(R)\subset R'$, continuity gives $h(\overline R)\subset\overline{h(R)}\subset\overline{R'}$. Hence the maps restrict to closures and descend to Hausdorff quotients. Naturality of homotopy exact sequences proves the rest. No assertion about arbitrary intersections on cycle spaces is used.
\end{proof}

\begin{corollary}\label{cor:lawsonbound}
If $\operatorname{rank} L_pH_{2p+1}(X)<\infty$, then $Q_p(X)$ is a finite dimensional compact real torus and
\[
 \dim_{\mathbb{R}}Q_p(X)=\operatorname{rank}\Lambda_p(X)\leq\operatorname{rank} L_pH_{2p+1}(X).
\]
If $L_pH_{2p+1}(X)\otimes\mathbb{Q}=0$, then $\operatorname{CR}_p(X)=G_p(X)$.
\end{corollary}
\begin{proof}
Tensor \eqref{eq:exact} with $\mathbb{Q}$ and use the free quotient $\Lambda_p(X)$ and the weak product description. A free group of rank zero is zero, and its corresponding weak product is trivial. Finite generation of the Lawson group is not required.
\end{proof}

\begin{proposition}\label{prop:lowerbound}
Let $B$ be a finite dimensional compact connected real torus. A continuous surjective homomorphism $\beta:G_p(X)\to B$ that vanishes on $R_p(X)$ induces a surjection
\[
 \Lambda_p(X)\otimes\mathbb{Q}\longrightarrow H_1(B,\mathbb{Q}).
\]
Thus $\operatorname{rank}\Lambda_p(X)\geq\dim_{\mathbb{R}}B$. If such targets have unbounded dimensions for one fixed $X,p$, then $\Lambda_p(X)$ has countably infinite rank.
\end{proposition}
\begin{proof}
The map kills $\operatorname{CR}_p(X)$ by continuity and descends to a continuous surjection $\bar\beta:Q_p(X)\to B$. The nested image tori $\bar\beta(\mathcal{A}_d)$ have bounded dimensions and eventually stabilize, because an inclusion of connected tori of the same dimension is equality. Their union is $B$, so some $\mathcal{A}_d$ surjects onto $B$. A surjection of finite dimensional tori has a full rank lattice map, possibly of finite index. Rationalizing gives the assertion. Countability follows from Theorem~\ref{thm:main}.
\end{proof}
Integral surjectivity is not claimed here: a nontrivial isogeny is the basic obstruction.

\subsection{Global triviality and integral homology}
\begin{lemma}\label{lem:global-section}
Let $q:G\to Q$ be a surjective homomorphism of topological abelian groups which is a Hurewicz fibration. Suppose its kernel $H$ is path connected and $Q$ is a finite or countable weak product of circles with its finite-stage final topology. Then $q$ has a global continuous section $s$ with $s(0)=0$, and
\begin{equation}\label{eq:generaltrivialization}
 H\times Q\xrightarrow{\ \sim\ } G,\qquad (h,t)\longmapsto h+s(t)
\end{equation}
is a homeomorphism and a trivialization of the principal $H$-bundle whenever that bundle structure is given. If $G$ is torsion free and $Q\neq0$, no section is a group homomorphism.
\end{lemma}
\begin{proof}
Write $Q=\mathop{\bigoplus}\nolimits_{j\in I}\mathbb{T}$ and let $\iota_j:\mathbb{T}\to Q$ be the coordinate inclusion. The path $\gamma_j(t)=\iota_j(t\bmod\mathbb{Z})$, $0\leq t\leq1$, lifts to a path $\widetilde\gamma_j$ in $G$ starting at zero. Its endpoint $h_j$ belongs to $H$. Choose a path $\delta_j$ in $H$ from $0$ to $h_j$. Then
\[
 t\longmapsto\widetilde\gamma_j(t)-\delta_j(t)
\]
starts and ends at zero and still maps to $\gamma_j$. It descends to a based continuous map $s_j:\mathbb{T}\to G$ with $qs_j=\iota_j$.

For a finite-support point $(t_j)$, set $s((t_j))=\sum_js_j(t_j)$. This is well defined since $s_j(0)=0$. On each finite coordinate torus it is a finite sum of continuous maps; the final topology therefore proves its global continuity. It satisfies $qs=\operatorname{id}$ and $s(0)=0$. The inverse of \eqref{eq:generaltrivialization} is $g\mapsto(g-s(q(g)),q(g))$, which is continuous. Translation in the first factor proves equivariance for the principal action.

A nonzero weak product of circles has nonzero torsion. A group section would be injective and send such torsion into the torsion free group $G$, a contradiction.
\end{proof}

The cycle group $Z_p(X)$ is free abelian on the irreducible $p$-dimensional subvarieties, so its subgroup $G_p(X)$ is torsion free. This verifies the final hypothesis when the lemma is used in Theorem~\ref{thm:main}. In particular, global topological triviality must not be written as a splitting of topological groups.

\begin{corollary}\label{cor:homologyproduct}
For every $k\geq0$ there are noncanonical integral isomorphisms
\begin{equation}\label{eq:homologyproduct}
 H_k(G_p(X),\mathbb{Z})\cong
 \bigoplus_{a=0}^k H_{k-a}(\operatorname{CR}_p(X),\mathbb{Z})\otimes_{\mathbb{Z}}
 \bigwedge\nolimits_{\mathbb{Z}}^{a}\Lambda_p(X).
\end{equation}
\end{corollary}
\begin{proof}
The homology of the weak product of circles is the exterior algebra on $\Lambda_p(X)$, as is seen on finite coordinate tori and then by passage to their union. It is free over $\mathbb{Z}$ in every degree. Apply integral K\"unneth to \eqref{eq:globalproduct}; all $\operatorname{Tor}^{\mathbb{Z}}_1$ terms vanish. The isomorphisms depend on the chosen section, even though the exact sequence \eqref{eq:exact} is canonical.
\end{proof}

\subsection{Hurewicz maps and higher Hodge structures}
\begin{corollary}\label{cor:hodgeclosure}
For every $m\geq1$, the Hurewicz homomorphism
\[
 h_m:\pi_m\operatorname{CR}_p(X)\longrightarrow
 H_m(\operatorname{CR}_p(X),\mathbb{Z})
\]
is split injective. For $m\geq2$, the canonical isomorphism \eqref{eq:higher} equips $\pi_m\operatorname{CR}_p(X)$ with a functorial structure of a filtered colimit of integral mixed Hodge structures. Proper pushforward and flat pullback preserve these structures. If $\mathbb{Z}(t)=(2\pi i)^t\mathbb{Z}$ denotes the Tate structure, the composite
\begin{equation}\label{eq:closurehodgecycle}
 \pi_m\operatorname{CR}_p(X)\xrightarrow{\ \sim\ }L_pH_{2p+m}(X)
 \xrightarrow{\ \Phi_{p,2p+m}\ }H_{2p+m}(X,\mathbb{Z})(-p)
\end{equation}
is a morphism of such colimits.
\end{corollary}
\begin{proof}
Set $H=\operatorname{CR}_p(X)$. Since $H$ is closed in $Z_p(X)$, its induced topology is final for the compact sets $H\cap K_d$: a subset closed on every $H\cap K_d$ is closed on every $K_d$, and hence closed in $Z_p(X)$. Thus the compact-filtration hypothesis of \cite[Lemma 3.5.9, pp.~102--103]{LF04} is satisfied. The statement there asserts injectivity, and its proof constructs a homomorphism
\[
 r_m:H_m(H,\mathbb{Z})\longrightarrow\pi_mH,
 \qquad r_m\circ h_m=\operatorname{id},
\]
using the summation maps from infinite symmetric products. Hence the Hurewicz map is split injective.

Definition~3.5.6 and Theorem~3.5.7(i)--(iii) of \cite[p.~100]{LF04} give filtered colimits of integral mixed Hodge structures on the homology of the cycle space, and then on Lawson homology through its Hurewicz image. Their proper and flat functoriality passes to $\pi_mH$, for $m\geq2$, through the canonical isomorphism \eqref{eq:higher}. This is a filtered-colimit structure; no finite-generation assertion is involved.

For $a\geq1$, the $s$-map in \cite[Proposition 3.5.10(ii) and its proof, p.~104]{LF04} is a morphism
\[
 L_aH_k(X)\longrightarrow L_{a-1}H_k(X)(-1).
\]
Set $k=2p+m$ and compose for $a=p,p-1,\ldots,1$. The composite is the cycle map by \cite[Definition 3.4.1(2), p.~93; Theorem 3.4.3, p.~94]{LF04}. Under Dold--Thom, its target is $H_{2p+m}(X,\mathbb{Z})(-p)$, which gives \eqref{eq:closurehodgecycle}; for $p=0$ the twist is zero.
\end{proof}
This construction concerns higher homotopy groups. It does not endow the full quotient lattice $\Lambda_p(X)$ with a mixed Hodge structure, nor does it construct such a structure on the homology of the closure in all degrees.

\section{A torsion criterion, applications, and periods}\label{sec:torsion}
For smooth projective $X$, a homomorphism $\rho:A_p(X)\to A(\mathbb{C})$ to an abelian variety is regular if its evaluation on every cycle correspondence over a smooth pointed parameter variety is algebraic; fibers are refined Gysin fibers, as in \cite[Section 3.1]{ACMV17}. In particular, for a Chow family $f:T\to C_{p,d}(X)$ over a smooth pointed variety $(T,t_0)$, the map
\[
 t\longmapsto\rho\bigl([f(t)-f(t_0)]\bigr)
\]
is a morphism. For singular $X$ we use this Chow-family condition as the definition of regularity. Signed families are treated by addition. The cycle realization is the composite $\beta:G_p(X)\to A_p(X)\xrightarrow{\rho}A(\mathbb{C})$.

\begin{lemma}\label{lem:regularcontinuity}
Every regular homomorphism in the preceding sense has continuous cycle realization. It therefore annihilates $D_p(X)$ on Chow classes and factors continuously through $Q_p(X)$ on cycles.
\end{lemma}
\begin{proof}
Fix a compact stage $K_d\cap G_p(X)$ of \eqref{eq:Kd}. Its presentation uses finitely many pairs of irreducible projective Chow components $V,W$ such that their differences are algebraically trivial. Choose smooth projective resolutions $T_V\to V$, $T_W\to W$, with base points. The pulled-back family maps are denoted by $f_V,f_W$. On $T_V\times T_W$ the realization is
\[
 \rho([f_V(t)-f_V(t_0)])-
 \rho([f_W(u)-f_W(u_0)])+
 \beta(f_V(t_0)-f_W(u_0)),
\]
a continuous map by regularity. In the usual smooth-target definition, the Chow family is represented by its graph cycle, whose refined fibers give the corresponding Chow classes; alternatively one may first flatten the graph by a projective modification and resolve. This does not change the proper surjective cover of the compact stage.

The chosen pairs are exactly those for which one, and therefore every, cycle difference is algebraically trivial. Their finite disjoint union maps continuously and surjectively onto $K_d\cap G_p(X)$. It is compact and the target is Hausdorff, so this is a quotient map. The displayed maps agree on fibers, since $\rho$ is defined on Chow classes. They therefore descend continuously to this compact stage. The final topology proves continuity on $G_p(X)$. Since $\beta$ vanishes on $R_p(X)$ and its target is Hausdorff, it vanishes on $\operatorname{CR}_p(X)$. The factorization and the assertion on $D_p(X)$ follow.
\end{proof}

\begin{theorem}\label{thm:torsioncriterion}
Let $X$ be connected projective, $0\leq p<\dim X$, and let $\rho:A_p(X)\to A(\mathbb{C})$ be a surjective regular homomorphism to a complex abelian variety. Let $\beta:G_p(X)\to A(\mathbb{C})$ be its cycle realization, continuous by Lemma~\ref{lem:regularcontinuity}. If $\rho$ is injective on torsion, then
\begin{equation}\label{eq:criterion}
 \operatorname{CR}_p(X)=\ker\beta,\qquad \bar\beta:Q_p(X)\xrightarrow{\ \sim\ } A(\mathbb{C})
\end{equation}
as topological groups. In particular,
\begin{equation}\label{eq:criterionexact}
 0\longrightarrow\pi_1\operatorname{CR}_p(X)\longrightarrow L_pH_{2p+1}(X)
 \xrightarrow{\ \beta_*\ }H_1(A,\mathbb{Z})\longrightarrow0.
\end{equation}
\end{theorem}
\begin{proof}
Continuity and rational invariance give $\operatorname{CR}_p(X)\subset\ker\beta$. Given $z\in\ker\beta$, use \eqref{eq:curvepresentation}, absorbing signs into the homomorphisms. Regularity and the universal property of each Jacobian show that the corresponding maps $J(D_i)\to A$ are algebraic homomorphisms. Their sum is an algebraic homomorphism $b:B=\prod_iJ(D_i)\to A$, and the tuple $t=([u_i])$ belongs to $\ker b$.

Torsion points are dense in each component of this algebraic subgroup, by the argument in Proposition~\ref{prop:zero-closure}. Choose torsion $t_\nu\to t$ in $\ker b$. Write $u_i=u_i^+-u_i^-$ and add fixed effective divisors $E_i$ to reach the Abel projective-bundle range. Local sections through $u_i^++E_i$, keeping $u_i^-+E_i$ fixed, lift the coordinates of $t_\nu$ to cycles $u_{i,\nu}\to u_i$. Thus
\[
 z_\nu=\sum_i\Gamma_{f_i}(u_{i,\nu})\longrightarrow z,
 \qquad [z_\nu]\in A_p(X)_{\mathrm{tors}},\qquad \rho([z_\nu])=0.
\]
Torsion injectivity gives $z_\nu\in R_p(X)$. All degree bounds were fixed first, so convergence takes place in a single bounded stage, proving the reverse inclusion.

The induced $\bar\beta$ is a continuous bijection. The nested image tori $\bar\beta(\mathcal{A}_d)$ eventually stabilize by bounded dimension and cover $A$, so one of them is $A$. Injectivity then forces $Q_p(X)=\mathcal{A}_d$. The compact-to-Hausdorff bijection is a homeomorphism. Formula \eqref{eq:criterionexact} is \eqref{eq:exact} transported through this isomorphism.
\end{proof}
The theorem identifies the specified map, not merely the abstract isomorphism class of its target. Path connectedness, local contractibility, and the higher isomorphisms already hold without its torsion hypothesis.

\Needspace{6\baselineskip}
\subsection{Canonical applications}
\begin{corollary}\label{cor:classical}
For smooth connected projective $X$ of dimension $n\geq1$, the canonical Albanese, Picard, and codimension two Abel--Jacobi maps induce topological group isomorphisms
\begin{equation}\label{eq:classical}
 \begin{split}
 Q_0(X)&\cong\operatorname{Alb}(X),\qquad Q_{n-1}(X)\cong\operatorname{Pic}^0(X),\\
 Q_{n-2}(X)&\cong J_a^3(X)\qquad(n\geq3).
 \end{split}
\end{equation}
Each closure is the corresponding cycle kernel, and $\operatorname{CR}_{n-1}(X)=R_{n-1}(X)$. Consequently
\[
 \Lambda_0(X)\cong H_1(X,\mathbb{Z})_{\mathrm{free}},\quad
 \Lambda_{n-1}(X)\cong H^1(X,\mathbb{Z}),\quad
 \Lambda_{n-2}(X)\cong H_1(J_a^3(X),\mathbb{Z})\ (n\geq3).
\]
\end{corollary}
\begin{proof}
Roitman's theorem supplies torsion injectivity for the Albanese homomorphism \cite{Roitman,Bloch}. For divisors the map $A_{n-1}(X)\cong\operatorname{Pic}^0(X)$ is the canonical isomorphism; its continuous cycle realization has kernel $R_{n-1}(X)$ \cite[Proposition 4.5 and Theorem 4.6]{Fri91}. In codimension two, Murre's torsion injectivity is recalled in \cite[Theorem 3.4(d)]{ACMV17}; continuity of the cycle realization is \cite[Proposition 3.5.12(iii), p.~105]{LF04}. These maps are surjective and regular, so Theorem~\ref{thm:torsioncriterion} applies. Their period lattices give the formulas.
\end{proof}
These are independently verified torsion theorems for the canonical maps, not a torsion-injectivity assertion for all regular homomorphisms. If $q(X)=h^{1,0}(X)=0$, then $\Lambda_0(X)=\Lambda_{n-1}(X)=0$ and $R_0(X)$ is dense in $G_0(X)$, with no restriction on the geometric genus.

\subsection{Isogenies and the necessity of the torsion hypothesis}
\begin{proposition}\label{prop:isogeny}
Suppose $\rho_0:A_p(X)\to A(\mathbb{C})$ satisfies Theorem~\ref{thm:torsioncriterion}, with cycle realization $\beta_0$. Let $f:A\to B$ be an isogeny, put $F=\ker f$, and set $\beta=f\beta_0$. Then
\begin{equation}\label{eq:isogeny}
 \ker\beta/\operatorname{CR}_p(X)\cong F,\qquad
 \pi_0\ker\beta\cong F,\qquad
 \operatorname{coker}\bigl(\beta_*:L_pH_{2p+1}(X)\to H_1(B,\mathbb{Z})\bigr)\cong F.
\end{equation}
The first is a topological group isomorphism with $F$ discrete. The inclusion of $\operatorname{CR}_p(X)$ as the identity component of $\ker\beta$ is an isomorphism on all based homotopy groups of positive degree.
\end{proposition}
\begin{proof}
The criterion identifies $Q_p(X)$ with $A$ and $\ker\beta_0$ with $\operatorname{CR}_p(X)$. Therefore $\ker\beta=\beta_0^{-1}(F)$, whose fibers are translates of this path-connected kernel. They are open and closed, and give its connected and path components. The quotient topology on the finite group is discrete; this also follows from the local sections. The isogeny covering has exact lattice sequence
\[
 0\to H_1(A,\mathbb{Z})\xrightarrow{f_*}H_1(B,\mathbb{Z})\to F\to0.
\]
Since $(\beta_0)_*$ is surjective, the stated cokernel follows. A based sphere or homotopy lies in the identity component, proving the remaining assertion.
\end{proof}

\begin{example}\label{ex:isogeny}
Let $X=E\times\mathbb{P}^1$ for an elliptic curve $E$, and take $p=1$. The canonical divisor map $\rho_0:A_1(X)\cong\operatorname{Pic}^0(X)\cong E$ has realization $\beta_0$. For $N\geq2$, $\rho=[N]\rho_0$ is surjective and regular with continuous realization $\beta=[N]\beta_0$, but its torsion kernel is $E[N]\cong(\mathbb{Z}/N\mathbb{Z})^2$. For $0\neq t\in E[N]$, the divisor
\[
 z_t=[\{t\}\times\mathbb{P}^1]-[\{0\}\times\mathbb{P}^1]
\]
satisfies $\beta_0(z_t)=t$ and $\beta(z_t)=0$. Hence
\[
 \operatorname{CR}_1(X)=R_1(X)=\ker\beta_0\subsetneq\ker\beta,
 \quad\beta_*:\mathbb{Z}^2\to\mathbb{Z}^2,	 v\mapsto Nv,
 \quad\pi_0\ker\beta\cong(\mathbb{Z}/N\mathbb{Z})^2.
\]
The zero map $\rho':A_1(X)\to0$ is also surjective and regular with continuous realization $\beta'=0$. Nevertheless, the divisor computation in Section~\ref{subsec:divisors} gives
\[
 \pi_1\operatorname{CR}_1(X)=0,\qquad
 \ker\bigl(\beta'_*:L_1H_3(X)=\mathbb{Z}^2\to0\bigr)=\mathbb{Z}^2.
\]
Thus even the fundamental group kernel formula can fail without the torsion hypothesis. Its survival under a finite isogeny is a separate fact: $f_*$ is injective, so $\ker\beta_* = \ker(\beta_0)_*$ in Proposition~\ref{prop:isogeny}.
\end{example}

\subsection{The period homomorphism}\label{subsec:periods}
Assume $X$ is smooth and $r=n-p\geq1$. The usual intermediate Jacobian is
\[
 J^{2r-1}(X)=
 \frac{H^{2r-1}(X,\mathbb{C})}{F^rH^{2r-1}(X,\mathbb{C})+H^{2r-1}(X,\mathbb{Z})_{\mathrm{free}}}.
\]
Here the lattice is mapped into complex cohomology. The image of $A_p(X)$ under the classical Abel--Jacobi homomorphism is the abelian subvariety $J_a^{2r-1}(X)$. We use the same notation $\operatorname{AJ}$ for this Chow-class map and its cycle realization only when the domain is displayed. Poincar\'e duality identifies the lattice of the full Jacobian with $H_{2p+1}(X,\mathbb{Z})_{\mathrm{free}}$; the lattice of its algebraic image embeds in this group.

Throughout the rest of the paper we distinguish
\begin{equation}\label{eq:phinotation}
 \Phi_{p,k}:L_pH_k(X)\longrightarrow H_k(X,\mathbb{Z}),\qquad
 \overline{\Phi}_{p,2p+1}:L_pH_{2p+1}(X)\longrightarrow H_{2p+1}(X,\mathbb{Z})_{\mathrm{free}}.
\end{equation}
The second is the first followed by the torsion free quotient. Period compatibility \cite[Proposition 5.4, Theorem 5.7, and Corollary 5.9]{Walker} says
\begin{equation}\label{eq:periodcompat}
 \bigl(H_1(J_a^{2r-1}(X),\mathbb{Z})\hookrightarrow H_{2p+1}(X,\mathbb{Z})_{\mathrm{free}}\bigr)\circ\operatorname{AJ}_*
 =\overline{\Phi}_{p,2p+1}.
\end{equation}
A loop of $p$-cycles sweeps out its Lawson trace, whose periods measure the endpoint change of a lifted Abel--Jacobi loop. This equality does not by itself imply integral surjectivity of $\operatorname{AJ}_*$; the Walker comparison retains that obstruction.

\section{Identifications and explicit computations}\label{sec:examples}
\subsection{Zero-cycles}
For smooth $X$, Corollary~\ref{cor:classical} and Dold--Thom identify the first homotopy map with $H_1(X,\mathbb{Z})\twoheadrightarrow H_1(X,\mathbb{Z})_{\mathrm{free}}$. Hence
\begin{equation}\label{eq:zerocalculation}
 \pi_0\operatorname{CR}_0(X)=0,\qquad\pi_1\operatorname{CR}_0(X)=H_1(X,\mathbb{Z})_{\mathrm{tors}},\qquad
 \pi_m\operatorname{CR}_0(X)=H_m(X,\mathbb{Z})\quad(m\geq2).
\end{equation}
There is no condition on the geometric genus.

\subsection{Divisors}\label{subsec:divisors}
For smooth connected $X$ of dimension $n$, the divisor theorem \cite[Theorem 4.6]{Fri91} gives
\[
 \pi_1G_{n-1}(X)=H^1(X,\mathbb{Z}),\qquad \pi_2G_{n-1}(X)=\mathbb{Z},
 \qquad\pi_mG_{n-1}(X)=0\quad(m\geq3).
\]
The first map to the Picard lattice is an isomorphism. Therefore
\begin{equation}\label{eq:divisorcalculation}
 \operatorname{CR}_{n-1}(X)=R_{n-1}(X)\simeq K(\mathbb{Z},2),\qquad
 \pi_m\operatorname{CR}_{n-1}(X)=
 \begin{cases}\mathbb{Z},&m=2,\\0,&m=0,1\text{ or }m\geq3.
 \end{cases}
\end{equation}
In particular this completely computes the nontrivial positive-dimensional case on a smooth surface, regardless of its geometric genus.

\begin{example}\label{ex:k3}
Let $S\subset\mathbb{P}^3$ be a smooth quartic surface. It is a K3 surface with $p_g=1$ and $\pi_1(S)=0$. If $h$ is the hyperplane class, $c_2(TS)=6h^2$ and $\int_Sh^2=4$, so the Euler characteristic is $24$. Poincar\'e duality and the universal coefficient theorem then give
\[
 H_2(S,\mathbb{Z})=\mathbb{Z}^{22},\quad H_4(S,\mathbb{Z})=\mathbb{Z},\quad H_1(S,\mathbb{Z})=H_3(S,\mathbb{Z})=0.
\]
In particular $H_2$ has no torsion, since it is controlled by $H_1$ through duality and universal coefficients. Thus
\begin{equation}\label{eq:k3}
 \operatorname{CR}_0(S)\simeq K(\mathbb{Z}^{22},2)\times K(\mathbb{Z},4),\qquad
 \operatorname{CR}_1(S)=R_1(S)\simeq K(\mathbb{Z},2).
\end{equation}
CW homotopy type upgrades the weak equivalences with these finite products to homotopy equivalences. The first formula concerns the closure, not the unclosed subgroup $R_0(S)$. Since $\operatorname{Alb}(S)=0$, Proposition~\ref{prop:zero-closure} also gives $\operatorname{CR}_0(S)=G_0(S)$. Mumford's theorem \cite{Mumford} gives $\operatorname{CH}_0(S)^0\neq0$, where the superscript denotes the kernel of degree. Hence
\[
 R_0(S)\subsetneq\operatorname{CR}_0(S),\qquad
 D_0(S)\cong\operatorname{CH}_0(S)^0\neq0,
 \qquad Q_0(S)=0.
\]
Thus even a zero Hausdorff quotient does not imply that rational equivalence is closed. This is an application of the classical zero-cycle theorem and the closure calculation, not a new construction of a K3 surface.
\end{example}

\subsection{Codimension two}
For smooth $X$, $n\geq3$, and $p=n-2$, the independent torsion theorem used in Corollary~\ref{cor:classical} gives
\begin{equation}\label{eq:codim2}
 \operatorname{CR}_{n-2}(X)=\ker\bigl(G_{n-2}(X)\xrightarrow{\operatorname{AJ}}J_a^3(X)\bigr).
\end{equation}
Using \eqref{eq:periodcompat}, we obtain
\begin{equation}\label{eq:codim2homotopy}
 \begin{split}
 \pi_0\operatorname{CR}_{n-2}(X)&=0,\\
 \pi_1\operatorname{CR}_{n-2}(X)&=\ker\bigl(\overline{\Phi}_{n-2,2n-3}:L_{n-2}H_{2n-3}(X)\to H_{2n-3}(X,\mathbb{Z})_{\mathrm{free}}\bigr),\\
 \pi_m\operatorname{CR}_{n-2}(X)&=L_{n-2}H_{2n-4+m}(X)\quad(m\geq2).
 \end{split}
\end{equation}
The codomain of the period map may be restricted to $H_1(J_a^3(X),\mathbb{Z})$, onto which it is surjective.

\begin{example}\label{ex:threefold}
For a smooth projective threefold and $p=1$,
\[
 \operatorname{CR}_1(X)=\ker\operatorname{AJ},\qquad \pi_1\operatorname{CR}_1(X)=\ker\overline{\Phi}_{1,3},\qquad
 \pi_m\operatorname{CR}_1(X)=L_1H_{m+2}(X)\ (m\geq2).
\]
Computing $\pi_2\operatorname{CR}_1(X)$ in singular-homology terms is precisely the additional comparison problem for $L_1H_4(X)$. The closure theorem neither proves that comparison nor imposes a universal vanishing range on the remaining Lawson groups.
\end{example}

\subsection{A surface times the projective line}
\begin{theorem}\label{thm:surfaceproduct}
Let $S$ be a smooth connected projective surface, $Y=S\times\mathbb{P}^1$, $\mathrm{pr}:Y\to S$ the projection, and $i_\infty:S\hookrightarrow Y$ the inclusion at infinity. The map
\begin{equation}\label{eq:surfaceproductmap}
 \Psi_{\operatorname{CR}}:\operatorname{CR}_0(S)\times\operatorname{CR}_1(S)\longrightarrow\operatorname{CR}_1(Y),\qquad
 (z,D)\longmapsto\mathrm{pr}^*z+i_{\infty*}D
\end{equation}
is a homotopy equivalence. In particular,
\begin{equation}\label{eq:surfaceproducttype}
 \operatorname{CR}_1(S\times\mathbb{P}^1)\simeq\operatorname{CR}_0(S)\times K(\mathbb{Z},2),
\end{equation}
and
\begin{equation}\label{eq:surfaceproducthomotopy}
 \pi_m\operatorname{CR}_1(S\times\mathbb{P}^1)=
 \begin{cases}
 H_1(S,\mathbb{Z})_{\mathrm{tors}},&m=1,\\
 H_2(S,\mathbb{Z})\oplus\mathbb{Z},&m=2,\\
 H_3(S,\mathbb{Z}),&m=3,\\
 \mathbb{Z},&m=4,\\
 0,&m=0\text{ or }m\geq5.
 \end{cases}
\end{equation}
\end{theorem}
\begin{proof}
The projective-line cycle-space formula makes the same map
\[
 \Psi:Z_0(S)\times Z_1(S)\longrightarrow Z_1(S\times\mathbb{P}^1)
\]
a homotopy equivalence by the cycle-space projective bundle theorem \cite[Proposition 2.5]{FG}; see also the projection--restriction diagram in \cite[Exercises 3.3.2(5), pp.~91--92]{LF04}. In the displayed order of factors, flat pullback applies to the zero-cycle and the section pushforward applies to the divisor, as in \eqref{eq:surfaceproductmap}. The group decomposition is recorded separately in \cite[Remark 3.3.15(1), p.~92]{LF04}. Proposition~\ref{prop:functoriality} shows that it restricts to the indicated closures. For $m\geq1$ its Lawson isomorphism is
\begin{equation}\label{eq:productlawson}
 L_1H_{m+2}(Y)\cong L_0H_m(S)\oplus L_1H_{m+2}(S).
\end{equation}
For $m\geq2$, naturality of \eqref{eq:higher} identifies $(\Psi_{\operatorname{CR}})_*$ with this isomorphism. Dold--Thom and the divisor calculation give the groups in \eqref{eq:surfaceproducthomotopy} in those degrees.

The first homotopy group requires its own verification. In degree three,
\[
 L_1H_3(Y)\cong H_1(S,\mathbb{Z})\oplus H^1(S,\mathbb{Z}),\qquad
 H_3(Y,\mathbb{Z})\cong\bigl(H_1(S,\mathbb{Z})\otimes H_2(\mathbb{P}^1,\mathbb{Z})\bigr)\oplus H_3(S,\mathbb{Z}).
\]
Under these decompositions, $\Phi^Y_{1,3}(a,b)=(a,\operatorname{PD}(b))$: the first term is the trace of a zero-cycle loop times $[\mathbb{P}^1]$, and the second is the divisor period identity. Here $\operatorname{PD}:H^1(S,\mathbb{Z})\xrightarrow{\ \sim\ } H_3(S,\mathbb{Z})$ and this group is torsion free. By \eqref{eq:codim2homotopy},
\[
 \pi_1\operatorname{CR}_1(Y)=\ker\overline{\Phi}^Y_{1,3}=H_1(S,\mathbb{Z})_{\mathrm{tors}}\oplus0.
\]
The zero-cycle fibration identifies $\pi_1\operatorname{CR}_0(S)$ with this torsion subgroup, and $\pi_1\operatorname{CR}_1(S)=0$. Thus $\Psi_{\operatorname{CR}}$ is also an isomorphism on $\pi_1$. All spaces are path connected and of CW homotopy type, so Whitehead's theorem proves the homotopy equivalence. Formulas \eqref{eq:zerocalculation} and \eqref{eq:divisorcalculation} give the stated type and groups.
\end{proof}

\begin{corollary}\label{cor:surfacequotient}
There are topological group and lattice isomorphisms
\begin{equation}\label{eq:surfacequotient}
 Q_1(S\times\mathbb{P}^1)\cong\operatorname{Alb}(S)\times\operatorname{Pic}^0(S),\qquad
 \Lambda_1(S\times\mathbb{P}^1)\cong H_1(S,\mathbb{Z})_{\mathrm{free}}\oplus H^1(S,\mathbb{Z}).
\end{equation}
In particular, $\operatorname{rank}\Lambda_1=2b_1(S)=4q(S)$.
\end{corollary}
\begin{proof}
The maps in \eqref{eq:surfaceproductmap} descend to a continuous homomorphism from $\operatorname{Alb}(S)\times\operatorname{Pic}^0(S)$ to $Q_1(S\times\mathbb{P}^1)$. The proof of Theorem~\ref{thm:surfaceproduct} and naturality of \eqref{eq:exact} identify its lattice map with the displayed isomorphism. The target lattice has finite rank, so its quotient is a finite dimensional torus. A homomorphism of such tori inducing an integral lattice isomorphism is an isomorphism. Hodge decomposition gives the rank formula.
\end{proof}

\subsection{Complex cellular varieties}
\begin{definition}\label{def:cellular}
A complex variety $Y$ is complex cellular if it has a finite filtration by closed subvarieties
\[
 \varnothing=Y_{-1}\subset Y_0\subset\cdots\subset Y_N=Y,
 \qquad Y_i\setminus Y_{i-1}\cong\mathbb{A}^{d_i}_{\mathbb{C}},
\]
where the isomorphisms are algebraic and $d_i\geq0$. The number of strata is not required to be $\dim Y+1$.
\end{definition}

One may equivalently allow a finite disjoint union of affine spaces at each step. Each component is open and closed in its stratum; adjoining the components successively refines the closed filtration to one affine space at a time.

\begin{remark}\label{rem:cellular}
The definition requires algebraic affine strata, not merely a CW structure or strata homeomorphic to affine spaces. Relative cellular spaces allow vector bundles over smooth projective bases \cite[Definition 6.1]{Karpenko}. For example, $E\times\mathbb{P}^1$ has relative cells $E$ and $E\times\mathbb{A}^1$ but is not complex cellular, since its first homology is $\mathbb{Z}^2$. Its divisor quotient is $E$, so the vanishing conclusion below does not extend to relative cellular spaces in general.
\end{remark}

For a smooth projective complex cellular $X$, the integral comparisons are
\begin{equation}\label{eq:cellularcomparison}
 \operatorname{CH}_p(X)\xrightarrow{\ \sim\ } H_{2p}(X,\mathbb{Z}),\qquad
 \Phi_{p,k}:L_pH_k(X)\xrightarrow{\ \sim\ } H_k(X,\mathbb{Z})\quad(k\geq2p).
\end{equation}
The Chow statement, freeness, and vanishing of odd homology are \cite[Examples 1.9.1 and 19.1.11(b)]{Fulton}. For the Lawson calculation, use \cite[Section 4(2)]{HuLi} with one configuration point, so $X[1]=X$; its proof uses localization from \cite{LF92}.

\begin{proposition}\label{prop:cellular}
If $X$ is smooth projective and complex cellular, then
\[
 G_p(X)=R_p(X)=\operatorname{CR}_p(X),\qquad Q_p(X)=0.
\]
For $0\leq p<n$, with $b_i=\operatorname{rank} H_i(X,\mathbb{Z})$,
\begin{equation}\label{eq:cellularhomotopy}
 \pi_m\operatorname{CR}_p(X)=
 \begin{cases}
 \mathbb{Z}^{b_{2p+2j}},&m=2j,	1\leq j\leq n-p,\\
 0,&m\text{ odd, }m=0,\text{ or }m>2(n-p).
 \end{cases}
\end{equation}
Its homotopy type is $\prod_{j=1}^{n-p}K(\mathbb{Z}^{b_{2p+2j}},2j)$.
\end{proposition}
\begin{proof}
Algebraic equivalence preserves integral cycle classes. The Chow isomorphism in \eqref{eq:cellularcomparison} therefore gives $G_p=R_p$. This component is closed, so it equals its closure and $Q_p=0$. The Lawson comparison gives \eqref{eq:cellularhomotopy}, and Corollary~\ref{cor:moore} together with CW homotopy type gives the finite product description.
\end{proof}

\begin{example}\label{ex:projectivespace}
For $X=\mathbb{P}^n$ and $0\leq p<n$, the only nonzero homotopy groups of $\operatorname{CR}_p(\mathbb{P}^n)$ are $\pi_{2j}=\mathbb{Z}$ for $1\leq j\leq n-p$. Its homotopy type is $\prod_{j=1}^{n-p}K(\mathbb{Z},2j)$.
\end{example}

\begin{example}\label{ex:grassmannian}
For the four-dimensional Grassmannian $X=\operatorname{Gr}(2,4)$ of two-planes in $\mathbb{C}^4$, the Schubert cells give $b_4=2$, $b_6=b_8=1$. At $p=1$ the only nonzero homotopy groups of the closure are
\[
 \pi_2\operatorname{CR}_1(X)=\mathbb{Z}^2,\qquad \pi_4\operatorname{CR}_1(X)=\mathbb{Z},\qquad \pi_6\operatorname{CR}_1(X)=\mathbb{Z}.
\]
The homotopy type is $K(\mathbb{Z}^2,2)\times K(\mathbb{Z},4)\times K(\mathbb{Z},6)$.
\end{example}

\subsection{Finite ranks and a varying-variety direct limit}
\begin{example}\label{ex:genus}
For $g\geq1$, let $B_g$ be the smooth projective model of $y^2=x^{2g+1}-1$, and put $X_g=B_g\times\mathbb{P}^1$. The double cover has $2g+2$ branch points, including infinity, so $B_g$ has genus $g$. One-cycles on $X_g$ are divisors and
\[
 Q_1(X_g)=\operatorname{Pic}^0(X_g)\cong J(B_g),\qquad\Lambda_1(X_g)=\mathbb{Z}^{2g}.
\]
Thus arbitrarily large finite ranks occur while the dimension of the variety remains two.
\end{example}

\begin{example}\label{ex:limit}
Fix an elliptic curve $E$ and, for every integer $a\geq1$, let $X_a=E^a\times\mathbb{P}^1$, $p=a$. Then $Q_a(X_a)=\operatorname{Pic}^0(X_a)\cong E^a$. Flat pullback along $X_{a+1}\to X_a$ is the standard inclusion with zero last coordinate. The final-topology direct limit is
\[
 Q_\infty\cong\mathop{\bigoplus}\nolimits_{j\geq1}E(\mathbb{C}),\qquad
 \pi_1Q_\infty=\bigoplus_{j\geq1}\mathbb{Z}^2,\qquad\pi_mQ_\infty=0\ (m\geq2).
\]
Both the variety and the cycle dimension vary. This does not realize an infinite rank $\Lambda_p(X)$ for one fixed $X,p$.
\end{example}

\subsection{An irreducible singular surface}
\begin{example}\label{ex:nodal}
Let $C\subset\mathbb{P}^2$ be the irreducible nodal cubic $y^2z=x^2(x+z)$, and let $X=C\times\mathbb{P}^1$. With ordinary Chow groups,
\[
 Q_0(C)=Q_1(X)=0,\qquad\operatorname{CR}_0(C)=G_0(C),\qquad\operatorname{CR}_1(X)=G_1(X).
\]
For either $F=\operatorname{CR}_0(C)$ or $F=\operatorname{CR}_1(X)$,
\[
 \pi_1F=\mathbb{Z},	\pi_2F=\mathbb{Z},	\pi_mF=0\ (m\geq3),	
 F\simeq K(\mathbb{Z},1)\times K(\mathbb{Z},2).
\]
Thus the smooth zero-cycle and divisor formulas do not extend to these singular varieties.
\end{example}
\begin{proof}
The normalization $\mathbb{P}^1\to C$ identifies two distinct points; hence $C\simeq S^2\vee S^1$. Every degree zero cycle on $C$ lifts to a degree zero cycle on $\mathbb{P}^1$, which is rationally trivial. Proper pushforward therefore gives $\operatorname{CH}_0(C)=\mathbb{Z}$ via degree, so $R_0(C)=G_0(C)$ and this group is closed. Dold--Thom computes its homotopy groups.

The integral projective bundle formula \cite[Theorem 3.3]{Fulton} gives
\[
 \operatorname{CH}_1(C\times\mathbb{P}^1)\cong\operatorname{CH}_0(C)\oplus\operatorname{CH}_1(C)\cong\mathbb{Z}^2.
\]
Its cycle map to $H_2(X,\mathbb{Z})=\mathbb{Z}^2$ is an isomorphism, with generators $[c\times\mathbb{P}^1]$ and $[C\times\{t\}]$. Algebraic equivalence preserves this class, so $G_1(X)=R_1(X)=\operatorname{CR}_1(X)$. The cycle-space projective bundle formula for singular bases \cite[Proposition 2.5]{FG} and the discreteness of $Z_1(C)=\mathbb{Z}[C]$ give
\[
 \pi_m\operatorname{CR}_1(X)\cong\pi_mZ_0(C)\oplus\pi_mZ_1(C)=H_m(C,\mathbb{Z})\quad(m\geq1).
\]
Theorem~\ref{thm:main} supplies path connectedness and CW homotopy type. Moore's splitting gives the displayed finite Eilenberg--Mac Lane products.
\end{proof}

\subsection{Top dimensional and disconnected varieties}\label{subsec:top}
For any reduced projective variety $X$ of dimension $n$, $Z_n(X)$ is the discrete free abelian group on its $n$-dimensional irreducible components. There are no nonzero algebraically or rationally trivial top-dimensional cycles; thus
\[
 G_n(X)=R_n(X)=\operatorname{CR}_n(X)=0,	 Q_n(X)=0,	\Lambda_n(X)=0.
\]
For $p>n$ the entire cycle group is zero. A zero-dimensional reduced projective variety is a finite set, with the same conclusion. If $X$ has finitely many connected components $X_i$, then for each $p$ its cycle groups, closures, and quotients are the finite products of those for $X_i$. The compact degree filtrations give the product topology. Consequently the structural results extend to disconnected $X$ componentwise.

\subsection{Projective stabilization}
\begin{corollary}\label{cor:projectivestabilization}
If $T$ is smooth connected projective with $\dim T\leq3$, then for every $a\geq0$ and $0\leq p\leq\dim T+a$,
\begin{equation}\label{eq:projectivestabilization}
 Q_p(T\times\mathbb{P}^a)\cong\prod_{i=0}^aQ_{p-i}(T),\qquad
 \Lambda_p(T\times\mathbb{P}^a)\cong\bigoplus_{i=0}^a\Lambda_{p-i}(T).
\end{equation}
In this formula $Q_j(T)=\Lambda_j(T)=0$ for $j<0$ or $j\geq\dim T$, consistently with the top-dimensional convention. All these lattices have finite rank.
\end{corollary}
\begin{proof}
The integral projective bundle formula gives $A_p(T\times\mathbb{P}^a)\cong\bigoplus_{i=0}^aA_{p-i}(T)$ \cite[Theorem 3.3]{Fulton}. Integral cohomological K\"unneth, with the appropriate Tate shifts, gives the corresponding decomposition of the classical Abel--Jacobi map and its algebraic image. Every nonzero summand on $T$ is a zero-cycle, divisor, or codimension two group. Its canonical map is injective on torsion by Corollary~\ref{cor:classical}. Hence the direct sum map is torsion-injective, and its cycle realization is continuous as the classical Abel--Jacobi map into the decomposed image. Theorem~\ref{thm:torsioncriterion} identifies the quotient with the product of the indicated Jacobians. Exterior product with linear spaces gives the displayed quotient maps; taking first homotopy groups proves the lattice formula.
\end{proof}

\subsection{Change under a threefold blow-up}
\begin{proposition}\label{prop:blowup}
Let $Y$ be a smooth projective threefold, let $C\subset Y$ be a smooth connected curve, and put $\widetilde Y=\operatorname{Bl}_CY$. Then
\begin{equation}\label{eq:blowup}
 Q_1(\widetilde Y)\cong Q_1(Y)\times J(C),\qquad
 \Lambda_1(\widetilde Y)\cong\Lambda_1(Y)\oplus H_1(C,\mathbb{Z}).
\end{equation}
Consequently $\operatorname{rank}\Lambda_1(\widetilde Y)=\operatorname{rank}\Lambda_1(Y)+2g(C)$.
\end{proposition}
\begin{proof}
The integral blow-up formulas give
\[
 \operatorname{CH}^2(\widetilde Y)_{\mathrm{alg}}\cong\operatorname{CH}^2(Y)_{\mathrm{alg}}\oplus\operatorname{Pic}^0(C),\qquad
 H^3(\widetilde Y,\mathbb{Z})\cong H^3(Y,\mathbb{Z})\oplus H^1(C,\mathbb{Z})(-1);
\]
see \cite[Proposition 6.7]{Fulton}. The formulas are compatible with Abel--Jacobi. The exceptional $\mathbb{P}^1$-fibers over $C$ give the whole $J(C)$ summand, hence $J_a^3(\widetilde Y)\cong J_a^3(Y)\times J(C)$. Corollary~\ref{cor:classical} proves the quotient formula, and taking first homotopy groups gives the lattice formula.
\end{proof}
For $\widetilde Y=\operatorname{Bl}_C\mathbb{P}^3$ this gives $\Lambda_1(\widetilde Y)\cong\mathbb{Z}^{2g(C)}$, whereas $\Lambda_1(\mathbb{P}^3)=0$. Thus the invariant is not birationally invariant. Its variation in Corollary~\ref{cor:variation} also rules out deformation invariance, semicontinuity, and constructibility in general.

\section{Integral Walker quotients and lattice criteria}\label{sec:invariants}
In this section $X$ is smooth and $0\leq p<n$, with $r=n-p$. Let
\[
 J_W^{2r-1}(X)=J\bigl(N^{r-1}H^{2r-1}(X,\mathbb{Z}(r))\bigr)
\]
be the Walker intermediate Jacobian. For a pure integral Hodge structure $H$ of weight $-1$, put $H_{\mathbb{C}}=H\otimes_{\mathbb{Z}}\mathbb{C}$ and write $J(H)=H_{\mathbb{C}}/(F^0H_{\mathbb{C}}+H_{\mathrm{free}})$, with its specified lattice. Here $\mathbb{Z}(r)=(2\pi i)^r\mathbb{Z}$, and the integral coniveau group is the union of the kernels of restriction to $X\setminus W$, over closed subsets $W$ of codimension at least $r-1$. Finite unions of supports make this a subgroup. Torsion is removed in forming its lattice; the indicated Tate twist gives a weight-minus-one Hodge structure.

Walker supplies a surjective regular homomorphism on Chow classes
\begin{equation}\label{eq:walkerfactor}
 \psi_W:A_p(X)\longrightarrow J_W^{2r-1}(X),\qquad
 \operatorname{AJ}=\alpha\psi_W,	\alpha:J_W^{2r-1}(X)\longrightarrow J_a^{2r-1}(X),
\end{equation}
where $\alpha$ is an isogeny \cite{Walker,ACMV22}; see also Suzuki's direct construction \cite[Theorem 1.1]{SuzukiFactor}. Write
\[
 \beta_W:G_p(X)\longrightarrow J_W^{2r-1}(X)
\]
for its realization on cycles. Regularity is \cite[Corollary 2.6 and Remark 2.7]{SuzukiFactor}; Lemma~\ref{lem:regularcontinuity} gives continuity on cycles. The finite dimensional Walker Jacobian here is distinct from the morphic Jacobian in \cite[Theorem 5.7]{Walker}. Thus $\beta_W$ descends to $\bar\beta_W:Q_p(X)\to J_W^{2r-1}(X)$, and
\begin{equation}\label{eq:walkerkernels}
 \operatorname{CR}_p(X)\subset\ker\beta_W\subset\ker\operatorname{AJ}.
\end{equation}
The final kernel is taken on $G_p(X)$. If $\psi_W$ is injective on torsion, Theorem~\ref{thm:torsioncriterion} makes the first inclusion equality. Suzuki's theorem says that the abstract Chow kernel $\ker\psi_W$ is divisible \cite[Theorem 2.3]{ACMV22}; it does not say that it is torsion free. The single prime criteria of Theorem~\ref{thm:primarycriterion} below are separate sufficient conditions.

\subsection{The invariant family and periods of curve families}
The family $(\Lambda_0(X),\ldots,\Lambda_n(X))$ is intrinsic to the algebraic variety with its cycle topology. Each member is free abelian, so its abstract isomorphism class is determined by its finite or countably infinite rank. The natural maps from Lawson homology and to period lattices contain further information. In particular, an abstract lattice does not recover a Jacobian's complex structure or polarization.

For smooth connected $X$, the established identifications are
\begin{center}
\begin{tabular}{@{}ll@{}}
\toprule
Cycle dimension & Quotient lattice \\
\midrule
$p=0$ & $H_1(X,\mathbb{Z})_{\mathrm{free}}$ \\
$p=n-1$ & $H^1(X,\mathbb{Z})$ \\
$p=n-2$, $n\geq3$ & $H_1(J_a^3(X),\mathbb{Z})$ \\
$p=n$ & $0$ \\
\bottomrule
\end{tabular}
\end{center}
Thus every entry in dimensions at most three has a classical interpretation and finite rank. At the remaining indices $1\leq p\leq n-3$, beginning with one-cycles on a fourfold, the Walker lattice and an additional lattice enter the calculation. None of the results below identifies that additional lattice for a general variety.

\begin{proposition}\label{prop:curvegeneration}
A signed projective algebraic family of $p$-cycles over a smooth connected projective curve $C$ induces continuous homomorphisms $J(C)\to Q_p(X)$ and $H_1(C,\mathbb{Z})\to\Lambda_p(X)$. Moreover,
\begin{equation}\label{eq:curvegeneration}
 \Lambda_p(X)\otimes\mathbb{Q}
 =\sum_{(C,\Gamma)}\operatorname{im}\bigl(H_1(C,\mathbb{Q})\longrightarrow\Lambda_p(X)\otimes\mathbb{Q}\bigr).
\end{equation}
Here a signed family means a finite signed sum of families represented by Chow morphisms; the associated graph families over the smooth curve may be taken flat.
\end{proposition}
\begin{proof}
The signed family map factors through $J(C)$ by Lemma~\ref{lem:familyfactor}. Every $\mathcal{A}_d$ is covered by a finite product of parameter Albaneses, each covered by a curve Jacobian. The resulting torus surjection is surjective on rational first homology. Passing to the union proves \eqref{eq:curvegeneration}.
\end{proof}
Integral generation by these particular Jacobian maps is not claimed. A torus surjection can have a finite-index image on integral lattices. The full integral lattice is determined by the primitive inclusions of the actual image tori $\mathcal{A}_d$.

\subsection{The canonical Walker quotient}\label{subsec:walkerquotient}
Assume $X$ is smooth, and retain the distinction between the Chow map $\psi_W$, its cycle realization $\beta_W$, and its quotient realization $\bar\beta_W$. Set
\[
 L_W(X)=H_1(J_W^{2r-1}(X),\mathbb{Z})
 \cong\bigl(N^{r-1}H^{2r-1}(X,\mathbb{Z}(r))\bigr)_{\mathrm{free}}.
\]
Poincar\'e duality and the Tate normalization give the period injection
\[
 \iota_W:L_W(X)\hookrightarrow H_{2p+1}(X,\mathbb{Z})_{\mathrm{free}}.
\]
For any abelian group $A$, write $A[N]=\ker(N:A\to A)$ and $A[\ell^\infty]=\bigcup_{a\geq1}A[\ell^a]$ for a prime $\ell$. In particular, $A_p(X)[\ell]$ consists of Chow classes $[z]$ with $\ell z\in R_p(X)$, not cycles with $\ell z=0$. Write $\mathbb{Z}_\ell$ for the $\ell$-adic integers and $\mathbb{Q}_\ell$ for their fraction field. The Pr\"ufer group is $\mathbb{Q}_\ell/\mathbb{Z}_\ell\cong\mathbb{Z}[1/\ell]/\mathbb{Z}$.

\begin{theorem}\label{thm:walkerquotient}
The descended Walker homomorphism $\bar\beta_W:Q_p(X)\to J_W^{2r-1}(X)$ is surjective. Its map on fundamental groups
\[
 \rho_W:\Lambda_p(X)\longrightarrow L_W(X)
\]
is surjective over $\mathbb{Z}$ and satisfies
\begin{equation}\label{eq:walkerperiod}
 \iota_W\rho_W\kappa_p=\overline{\Phi}_{p,2p+1}.
\end{equation}
Set $K_p^{\mathrm{per}}(X)=\ker\rho_W$ and $\mathcal{T}_p(X)=\ker\bar\beta_W$. Then $\mathcal{T}_p(X)$ is a finite or countable weak product of circles, and there are noncanonical splittings
\begin{equation}\label{eq:walkersplitting}
 Q_p(X)\cong J_W^{2r-1}(X)\times\mathcal{T}_p(X),\qquad
 \Lambda_p(X)\cong L_W(X)\oplus K_p^{\mathrm{per}}(X).
\end{equation}
The first is a topological group isomorphism. The residual lattice has the canonical description
\begin{equation}\label{eq:residuallattice}
 K_p^{\mathrm{per}}(X)\cong
 \frac{\ker\overline{\Phi}_{p,2p+1}}{\operatorname{im}\bigl(\pi_1\operatorname{CR}_p(X)\longrightarrow L_pH_{2p+1}(X)\bigr)}.
\end{equation}
Abstractly, $\mathcal{T}_p(X)\cong\ker\psi_W/D_p(X)$.
\end{theorem}
\begin{proof}
Continuity and rational invariance give the descended surjection. The equality \eqref{eq:walkerperiod} follows from \eqref{eq:periodcompat} and $\operatorname{AJ}=\alpha\beta_W$ on cycles: the composite lattice map for $\alpha$ and inclusion of the algebraic Jacobian is exactly $\iota_W$.

Let $N_{p+1}H_{2p+1}(X,\mathbb{Z})$ denote the sum of the images of $H_{2p+1}(V,\mathbb{Z})$ for closed subvarieties $V\subset X$ with $\dim V\leq p+1$. Walker proves the integral equality
\[
 \operatorname{im}\Phi_{p,2p+1}=N_{p+1}H_{2p+1}(X,\mathbb{Z})
\]
in \cite[Proposition 2.8, p.~915]{Walker}. For smooth $X$, Poincar\'e duality with supports identifies this subgroup with the integral coniveau subgroup in the definition of $J_W^{2r-1}(X)$. Its image in $H_{2p+1}(X,\mathbb{Z})_{\mathrm{free}}$ is therefore exactly $\iota_W(L_W(X))$. In particular, one must not replace this image by its saturation. Formula~\eqref{eq:walkerperiod} and injectivity of $\iota_W$ show that $\rho_W\kappa_p$ is surjective onto $L_W(X)$; hence $\rho_W$ is surjective over $\mathbb{Z}$. This argument uses the integral theorem, not only a rational coniveau comparison.

Choose a basis of the finite free group $L_W(X)$ and lift it to $\Lambda_p(X)$. All lifts occur in $H_1(\mathcal{A}_e,\mathbb{Z})$ for some $e$. The map $\mathcal{A}_e\to J_W^{2r-1}(X)$ is surjective on lattices. Splitting that lattice map and extending the section real linearly gives a continuous group section
\[
 \sigma:J_W^{2r-1}(X)\longrightarrow \mathcal{A}_e\subset Q_p(X).
\]
The map $t\mapsto t-\sigma\bar\beta_W(t)$ now gives the first product decomposition in \eqref{eq:walkersplitting}. For $d\geq e$, restriction yields
\[
 \mathcal{A}_d\cong\bigl(\mathcal{A}_d\cap\mathcal{T}_p(X)\bigr)\times J_W^{2r-1}(X).
\]
The first factor is a compact connected real torus: the kernel in this product is connected because $\mathcal{A}_d$ is connected. These factors exhaust the closed subgroup $\mathcal{T}_p(X)$ with its induced topology. Indeed, a subset closed on each of these factors is closed on each $\mathcal{A}_d$, hence closed in $Q_p(X)$ by its final topology. Lemma~\ref{lem:tori} applies and proves the claimed weak product structure and lattice splitting.

Finally, injectivity of $\iota_W$, \eqref{eq:walkerperiod}, and \eqref{eq:exact} give \eqref{eq:residuallattice}. Since $Q_p=A_p/D_p$, its Walker kernel is abstractly $\ker\psi_W/D_p$. These are Chow-class quotients, not identifications of the unclosed rational equivalence subgroup with a topological fiber.
\end{proof}

The group splitting \eqref{eq:walkersplitting} must be distinguished from the space splitting \eqref{eq:globalproduct}. The former splits a torus-lattice map; the latter splits a quotient of the torsion free group of cycles and cannot be a group splitting when its quotient is nonzero.

The Walker factor has its specified Hodge structure; a pure Hodge structure on the full lattice $\Lambda_p$ is not furnished by the topological construction. The rank identities and bound are
\begin{equation}\label{eq:rankwalker}
 \begin{split}
 \operatorname{rank}\Lambda_p(X)&=\dim_{\mathbb{Q}}N^{r-1}H^{2r-1}(X,\mathbb{Q})+\operatorname{rank} K_p^{\mathrm{per}}(X),\\
 \dim_{\mathbb{Q}}N^{r-1}H^{2r-1}(X,\mathbb{Q})&=2\dim_{\mathbb{C}}J_W^{2r-1}(X)\leq2h^{r-1,r}(X).
 \end{split}
\end{equation}
The inequality bounds the Walker factor only, not the possible residual lattice.

\begin{corollary}\label{cor:walkerkernel}
The cycle level Walker kernel $W_p(X)=\ker\beta_W\subset G_p(X)$ is path connected, locally contractible, and of CW homotopy type. Its homotopy groups are
\begin{equation}\label{eq:walkerkernelhomotopy}
 \pi_1W_p(X)=\ker\overline{\Phi}_{p,2p+1},\qquad
 \pi_mW_p(X)=L_pH_{2p+m}(X)\quad(m\geq2).
\end{equation}
If $F=\ker(\alpha:J_W^{2r-1}(X)\to J_a^{2r-1}(X))$, then
\begin{equation}\label{eq:AJcomponents}
 \pi_0\ker\bigl(\operatorname{AJ}:G_p(X)\to J_a^{2r-1}(X)\bigr)\cong F,
 \qquad\operatorname{coker}\operatorname{AJ}_*\cong F.
\end{equation}
More precisely, there are noncanonical homeomorphisms of spaces
\begin{equation}\label{eq:kernelproducts}
 \begin{split}
 W_p(X)&\cong\operatorname{CR}_p(X)\times\mathcal{T}_p(X),\\
 \ker\operatorname{AJ}&\cong\operatorname{CR}_p(X)\times\mathcal{T}_p(X)\times F,
 \end{split}
\end{equation}
where $F$ is discrete and the classical kernel is taken on $G_p(X)$. Moreover, put
\[
 H=H_{2p+1}(X,\mathbb{Z}),\qquad
 N=\operatorname{im}\Phi_{p,2p+1},\qquad
 \overline H=H/H_{\mathrm{tors}},\qquad
 \overline N=\operatorname{im}(N\longrightarrow\overline H).
\]
The saturation $\overline N^{\mathrm{sat}}=\overline H\cap(\overline N\otimes\mathbb{Q})$ is taken inside $\overline H\otimes\mathbb{Q}$. There is a canonical isomorphism
\begin{equation}\label{eq:saturationcomponents}
 F\cong \overline N^{\mathrm{sat}}/\overline N.
\end{equation}
Consequently, the classical cycle level Abel--Jacobi kernel is connected if and only if $\overline N$ is saturated in $\overline H$.
\end{corollary}
\begin{proof}
Compose the section $\sigma$ of Theorem~\ref{thm:walkerquotient}, locally near zero, with a local section of $q_p$. This gives a local section of $\beta_W$. Its translates make $\beta_W$ a numerable principal $W_p$-bundle over the compact torus. The retraction argument of Theorem~\ref{thm:main} gives local contractibility, and \cite[Proposition 3]{Schon} gives CW homotopy type. Its first homotopy map is $\rho_W\kappa_p$, which is surjective. The homotopy exact sequence proves path connectedness and \eqref{eq:walkerkernelhomotopy}.

The inverse image of each $t\in F$ is a translate of $W_p$; these are open and closed in the classical kernel and are precisely its components. On lattices, $\operatorname{AJ}_* = \alpha_*\rho_W\kappa_p$, with the last two maps surjective. The isogeny lattice map has cokernel $F$, proving \eqref{eq:AJcomponents}.

For \eqref{eq:kernelproducts}, restrict the global trivialization of $q_p$ to the inverse image of $\mathcal{T}_p(X)\subset Q_p(X)$. Since $\beta_W=\bar\beta_Wq_p$, this inverse image is $W_p(X)$ and gives the first homeomorphism. Under the group splitting \eqref{eq:walkersplitting}, $\bar\beta_W^{-1}(F)$ is identified with $F\times\mathcal{T}_p(X)$. Restricting the same global trivialization gives the second homeomorphism. These last splittings are assertions about spaces, not about their group laws.

For \eqref{eq:saturationcomponents}, Walker's exact sequence \cite[Corollary 5.9 and Remark 5.10, p.~932]{Walker} identifies $F$ with the cokernel of
\[
 H_{\mathrm{tors}}\longrightarrow (H/N)_{\mathrm{tors}}.
\]
The quotient of $H/N$ by the image of $H_{\mathrm{tors}}$ is $\overline H/\overline N$. Since that image is finite, every torsion element of this latter quotient lifts to a torsion element of $H/N$. Hence the cokernel is canonically $(\overline H/\overline N)_{\mathrm{tors}}=\overline N^{\mathrm{sat}}/\overline N$. The connectedness assertion now follows from \eqref{eq:AJcomponents}. Without a torsion freeness assumption on $H$, the group $(H/N)_{\mathrm{tors}}$ alone is not the component group.
\end{proof}

\subsection{A comparison criterion}
\begin{proposition}[A Lawson comparison criterion]\label{prop:comparisoncriterion}
Let $X$ be smooth connected projective and $0\leq p<n$. If
\[
 \Phi_{p,2p+1}\otimes\mathbb{Q}:
 L_pH_{2p+1}(X)_{\mathbb{Q}}\longrightarrow H_{2p+1}(X,\mathbb{Q})
\]
is injective, then the descended Walker map is an isomorphism of topological groups. In particular,
\begin{equation}\label{eq:comparisoncriterion}
 \begin{gathered}
 Q_p(X)\cong J_W^{2r-1}(X),\qquad
 \Lambda_p(X)\cong L_W(X),\qquad
 \operatorname{CR}_p(X)=\ker\beta_W,\\
 \pi_1\operatorname{CR}_p(X)=\ker\overline{\Phi}_{p,2p+1}
 \quad\text{is a torsion group}.
 \end{gathered}
\end{equation}
Moreover, $D_p(X)=\ker\psi_W$ is divisible. Torsion freeness of this group is not asserted under the stated hypothesis.
\end{proposition}
\begin{proof}
Rational injectivity makes $\ker\overline{\Phi}_{p,2p+1}$ a torsion group. Its quotient $K_p^{\mathrm{per}}(X)$ in \eqref{eq:residuallattice} is free abelian, hence zero. The splitting \eqref{eq:walkersplitting} then gives the asserted isomorphisms and equality of cycle kernels. The fundamental group assertion follows from \eqref{eq:exact}. Suzuki's divisibility theorem \cite[Theorem 2.3]{ACMV22} gives divisibility of $D_p(X)$.

\end{proof}
\subsection{Classical Abel--Jacobi kernels}\label{sec:obstruction}
\begin{proposition}\label{prop:strictAJ}
There is a smooth projective fourfold $Y$ with
\[
 \operatorname{CR}_1(Y)\subsetneq\ker\bigl(G_1(Y)\xrightarrow{\operatorname{AJ}}J_a^5(Y)\bigr).
\]
\end{proposition}
\begin{proof}
Ottem and Suzuki construct a smooth projective threefold $T$ which is a pencil of Enriques surfaces and an elliptic curve $E$ such that, for $Y=T\times E$, the isogeny $\alpha:J_W^5(Y)\to J_a^5(Y)$ has nonzero kernel \cite[Corollary 4.2]{OS}. Choose $0\neq t\in\ker\alpha$. Surjectivity of $\psi_W$ gives an algebraically trivial cycle $z$ with $\beta_W(z)=t$. Then $\operatorname{AJ}(z)=0$, whereas every element of $\operatorname{CR}_1(Y)$ has Walker image zero by \eqref{eq:walkerkernels}. Thus $z\notin\operatorname{CR}_1(Y)$.
\end{proof}

\begin{corollary}\label{cor:disconnectedexample}
In this fourfold example, the cycle level classical Abel--Jacobi kernel is disconnected, and
\[
 \operatorname{AJ}_*:L_1H_3(Y)\longrightarrow H_1(J_a^5(Y),\mathbb{Z})
\]
is not surjective. Equivalently, the image of the integral cycle map in $H_3(Y,\mathbb{Z})_{\mathrm{free}}$ is not saturated.
\end{corollary}
\begin{proof}
Corollary~\ref{cor:walkerkernel}, including \eqref{eq:saturationcomponents}, identifies both the component group and the lattice cokernel with the nonzero finite kernel of the isogeny in Proposition~\ref{prop:strictAJ}.
\end{proof}

This is an integral obstruction: replacing the target by an isogenous torus can change the cycle kernel and its components. The local and global section arguments do not discard the finite isogeny kernel.

\subsection{Finite parameterization and torsion criteria}
\begin{theorem}\label{thm:finiterank}
For fixed connected projective $X$, allowing singularities, and $0\leq p<\dim X$, the following conditions are equivalent:
\begin{enumerate}
\item $\operatorname{rank}\Lambda_p(X)<\infty$;
\item $Q_p(X)$ is compact;
\item $Q_p(X)$ has finite covering dimension;
\item $q_p(K_d\cap G_p(X))=Q_p(X)$ for some $d$;
\item one signed projective algebraic family over a smooth connected projective curve $C$ induces a surjection $J(C)\to Q_p(X)$.
\end{enumerate}
Under these conditions $Q_p(X)$ is a compact real torus of dimension $\operatorname{rank}\Lambda_p(X)$, and the rank in (v) is at most $2g(C)$.
\end{theorem}
\begin{proof}
The weak product description gives equivalence of (i)--(iii): finite rank gives a compact finite dimensional torus, whereas infinite rank gives closed tori of arbitrarily large dimension and cannot be compact. Indeed, compact subsets lie in finite stages of its CW structure. Under (ii), some $\mathcal{A}_d$ is all of $Q_p$; its bounded presentation in Proposition~\ref{prop:exhaustion} gives (iv). Conversely, the compact image in (iv) gives (ii).

When $Q_p=\mathcal{A}_d$, finitely many curve Jacobians cover its parameter Albaneses. Choose a smooth ample complete intersection curve in the product of those parameter curves; its Jacobian surjects onto the product Jacobian. Pull the signed families back along the projections and sum them, giving (v). For one factor take that curve; for the zero quotient take $C=\mathbb{P}^1$ and the zero family. Condition (v) implies compactness, and the dimension of its source gives the genus bound.
\end{proof}
This finite-parameterization property concerns the residual group $A_p/D_p$. It neither bounds all discrete algebraic-equivalence components at once nor determines the abstract group $D_p$ that was removed.

\begin{lemma}\label{lem:primarysurjection}
For any connected projective $X$, allowing singularities, $0\leq p<\dim X$, and every prime $\ell$, there is an exact sequence of abstract groups
\begin{equation}\label{eq:primaryexact}
 0\longrightarrow D_p(X)[\ell^\infty]\longrightarrow A_p(X)[\ell^\infty]
 \longrightarrow Q_p(X)[\ell^\infty]\longrightarrow0.
\end{equation}
Moreover,
\[
 Q_p(X)[\ell^\infty]\cong\bigoplus_{j\in I}\mathbb{Q}_\ell/\mathbb{Z}_\ell,
 \qquad |I|=\operatorname{rank}\Lambda_p(X).
\]
The direct sum has finite support even when $I$ is infinite.
\end{lemma}
\begin{proof}
A torsion point of $Q_p$ lies in an $\mathcal{A}_d$ covered by a finite product $B$ of curve Jacobians. A surjection $B\to \mathcal{A}_d$ lifts torsion to torsion, possibly of larger order. To see this, let $t$ have order $N$ and choose a lift $b$. The element $Nb$ lies in the closed kernel. If $M$ kills the finite component group of that kernel, then $MNb$ is in its identity torus. Choose $y$ in that torus with $MNy=MNb$. Then $b-y$ is a torsion lift. If $t$ is $\ell$-primary, take the $\ell$-primary component of this lift. The corresponding curve Jacobian classes map to an $\ell$-primary class in $A_p(X)$ lifting $t$. The kernel is exactly $D_p[\ell^\infty]$. Taking torsion in the weak product of circles proves the final assertion.
\end{proof}
The exponent $\ell^\infty$ in this lemma cannot be replaced without further argument by a prescribed exponent $\ell^a$.

\begin{theorem}\label{thm:primarycriterion}
Let $X$ be connected projective, $0\leq p<\dim X$, and let $\ell$ be a prime. Part (i) permits singular $X$; in (ii)--(iii), assume $X$ smooth.
\begin{enumerate}
\item If $a_\ell=\dim_{\mathbb{F}_\ell}A_p(X)[\ell]<\infty$, then $\operatorname{rank}\Lambda_p(X)\leq a_\ell$.
\item If $\psi_W$ is injective on $A_p(X)[\ell]$, then
\[
 \operatorname{CR}_p(X)=\ker\beta_W,\qquad Q_p(X)\cong J_W^{2r-1}(X)
\]
via the induced Walker map.
\item The same conclusion holds if the kernel of the classical Abel--Jacobi map on $A_p(X)[\ell^\infty]$ has bounded exponent; in particular, it holds if this kernel is finite.
\end{enumerate}
\end{theorem}
\begin{proof}
The curve presentation makes $A_p(X)$ a sum of images of complex points of curve Jacobians, so it is divisible, also for singular $X$; compare \cite[Example 19.1.2]{Fulton}. If $a_\ell<\infty$, the classification of divisible torsion groups gives $A_p[\ell^\infty]\cong(\mathbb{Q}_\ell/\mathbb{Z}_\ell)^{a_\ell}$. Project the surjection in Lemma~\ref{lem:primarysurjection} onto any $t$ coordinate factors of the target. A homomorphism between finite sums of Pr\"ufer $\ell$-groups is a matrix over $\mathbb{Z}_\ell$, since $\operatorname{End}(\mathbb{Q}_\ell/\mathbb{Z}_\ell)=\mathbb{Z}_\ell$. Surjectivity forces row rank $t$ over $\mathbb{Q}_\ell$, so $t\leq a_\ell$. The number of possible coordinate factors is bounded by $a_\ell$, proving (i).

For smooth $X$, the group $K=\ker\psi_W$ is divisible by Suzuki's theorem \cite[Theorem 2.3]{ACMV22}. Consequently
\[
 0\to K[\ell]\to A_p(X)[\ell]\to J_W^{2r-1}(X)[\ell]\to0
\]
is exact. Under (ii), $K[\ell]=0$, so $a_\ell=\operatorname{rank} L_W(X)$. Part (i) and the integral surjection $\rho_W$ imply $\operatorname{rank}\Lambda_p=\operatorname{rank} L_W$. The residual free lattice in Theorem~\ref{thm:walkerquotient} is therefore zero, and so is its weak product torus $\mathcal{T}_p$. This proves (ii). In (iii), the divisible group $K[\ell^\infty]$ lies in a bounded-exponent group, hence is zero. Part (ii) applies.
\end{proof}
For smooth $X$, an infinite rank $\Lambda_p(X)$ would force both $A_p(X)[\ell]$ and $(\ker\psi_W)[\ell]$ to be infinite for every prime $\ell$. For singular $X$, the former conclusion still follows. These are necessary conditions, not converses.

The primary torsion sequence above does not imply surjectivity on every fixed order. It also does not determine the uniquely divisible part of $D_p(X)$.

\section{Coniveau ranks in a family of abelian threefolds}\label{sec:variation}
\subsection{The rank calculation}\label{subsec:abelianranks}
Write $\operatorname{End}^0(A)=\operatorname{End}(A)\otimes\mathbb{Q}$ and $\varrho(A)=\operatorname{rank}\operatorname{NS}(A)$, where $\operatorname{NS}(A)$ is the N\'eron--Severi group. An elliptic curve is non-CM when its rational endomorphism algebra is $\mathbb{Q}$, and has complex multiplication (CM) when that algebra is an imaginary quadratic field.

The following computations express classical coniveau subspaces through the quotient lattice. We use the Hodge groups in \cite{MZ}; the product statement for pairwise nonisogenous elliptic curves is due to Imai \cite{Imai}, as recalled in \cite[Corollary 3.9]{MZ}.

\begin{proposition}\label{prop:abelianranks}
The following abelian threefolds have the indicated ranks:
\begin{center}
\begin{tabular}{@{}>{\raggedright\arraybackslash}p{0.67\linewidth}rr@{}}
\toprule
Abelian threefold $A$ & $\varrho(A)$ & $\operatorname{rank}\Lambda_1(A)$ \\
\midrule
Simple $A$ with $\operatorname{End}^0(A)=\mathbb{Q}$ & $1$ & $6$ \\
$E_1\times E_2\times E_3$, pairwise nonisogenous non-CM elliptic curves & $3$ & $12$ \\
$E^2\times F$, with $E,F$ nonisogenous non-CM elliptic curves & $4$ & $14$ \\
$E^3$, with $E$ non-CM & $6$ & $16$ \\
$E^3$, with $E$ having complex multiplication & $9$ & $18$ \\
\bottomrule
\end{tabular}
\end{center}
All have $h^{a,b}(A)=\binom3a\binom3b$, $b_3(A)=20$, and underlying differentiable manifold $(S^1)^6$. In these five cases,
\begin{equation}\label{eq:divisorconiveau}
 N^1H^3(A,\mathbb{Q})=\operatorname{im}\bigl(\operatorname{NS}(A)_{\mathbb{Q}}\otimes H^1(A,\mathbb{Q})\xrightarrow{\ \wedge\ }H^3(A,\mathbb{Q})\bigr).
\end{equation}
\end{proposition}
\begin{proof}
Corollary~\ref{cor:classical} and Theorem~\ref{thm:walkerquotient} identify $\Lambda_1(A)$ with the integral coniveau lattice of $H^3(A)$. Thus it suffices to compute $\dim_{\mathbb{Q}}N^1H^3$. Cup product with a divisor class is supported on a divisor, so the right-hand side of \eqref{eq:divisorconiveau} is contained in the left-hand side. The reverse inclusion will follow from the Hodge groups described in \cite[Sections 1.2, 2.1, 2.3 and Corollary 3.9]{MZ}. A coniveau-one sub-Hodge structure in degree three has only types $(2,1)$ and $(1,2)$.

For simple $A$ with $\operatorname{End}^0(A)=\mathbb{Q}$, its Hodge group is $\operatorname{Sp}_6$. If $U_A=H^1(A,\mathbb{Q})$ and $\theta$ is an ample divisor class, the symplectic decomposition is
\[
 H^3(A,\mathbb{Q})=\theta\wedge U_A\oplus P^3,
 \qquad\dim(\theta\wedge U_A)=6,\quad\dim P^3=14.
\]
The primitive summand $P^3$ is irreducible and contains $H^{3,0}$. A nonzero projection of $N^1H^3$ onto it would be all of $P^3$, contradicting the Hodge types of coniveau one. Hence $N^1H^3=\theta\wedge U_A$. The divisor invariant subspace is $\mathbb{Q}\theta$, so $\varrho=1$. Such varieties exist, for example as very general principally polarized abelian threefolds.

For pairwise nonisogenous non-CM elliptic curves, write $U_i=H^1(E_i,\mathbb{Q})$. Their product Hodge group is $\operatorname{SL}_2^3$, and
\[
 H^3(A,\mathbb{Q})=
 \bigoplus_{i\neq j}\bigl(H^2(E_i,\mathbb{Q})\otimes U_j\bigr)
 \oplus(U_1\otimes U_2\otimes U_3).
\]
The first part has dimension $12$ and is generated by the coordinate divisors cupped with $H^1(A)$. The final tensor product is irreducible and contains type $(3,0)$, so its projection from $N^1H^3$ is zero. The divisor invariant subspace has dimension $3$.

Suppose $A=E^2\times F$, where $E,F$ are nonisogenous non-CM elliptic curves. Set $U=H^1(E,\mathbb{Q})$ and $V=H^1(F,\mathbb{Q})$. Its Hodge group is $\operatorname{SL}(U)\times\operatorname{SL}(V)$, acting diagonally on the two copies of $U$; this follows from \cite[Section 1.2 and Corollary 3.9]{MZ}. The K\"unneth decomposition has a $12$-dimensional summand $B$ generated by the coordinate divisor classes cupped with $H^1(A)$, and the remaining summand decomposes as
\[
 U\otimes U\otimes V
 =\bigl(\operatorname{Sym}^2U\otimes V\bigr)
   \oplus\bigl(\textstyle\bigwedge^2U\otimes V\bigr).
\]
The first term is an irreducible six-dimensional Hodge structure containing type $(3,0)$, so its projection from $N^1H^3(A,\mathbb{Q})$ is zero. The second term is $V(-1)$ and is generated by the cross-divisor class between the two $E$ factors cupped with $H^1(F)$. It is disjoint from $B$. Hence the divisor-product subspace has dimension $12+2=14$ and equals $N^1H^3(A,\mathbb{Q})$. The three coordinate divisor classes and this one cross-divisor class are the degree-two Hodge invariants, giving $\varrho(A)=4$.

For the non-CM cube, let $U_E=H^1(E,\mathbb{Q})$ and $M=\mathbb{Q}^3$ of type $(0,0)$. The diagonal $\operatorname{SL}_2$ action gives
\[
 H^3(E^3,\mathbb{Q})=\bigwedge^3(U_E\otimes M)
 \cong\operatorname{Sym}^3U_E\oplus U_E(-1)^{\oplus8}.
\]
Indeed, the weights $3,1,-1,-3$ have multiplicities $1,9,9,1$; subtracting the character of $\operatorname{Sym}^3U_E$ leaves eight standard representations with the displayed Tate twist. The first summand has dimension $4$, is irreducible, and contains type $(3,0)$. The other has dimension $16$. Divisor classes form the six-dimensional space $(\bigwedge^2U_E)\otimes\operatorname{Sym}^2M$. To check that their products span the full second summand, choose bases $a_i,b_i$ of the three copies of $U_E$. The invariant divisor classes are spanned by
\[
 a_i\wedge b_i\ (1\leq i\leq3),\qquad
 a_i\wedge b_j-b_i\wedge a_j\ (i<j).
\]
The coordinate divisors times $H^1$ span the $12$ degree-three basis monomials containing both $a_i,b_i$ for some $i$. Modulo these, the cross terms span two independent differences among the three monomials with two $a$'s and one $b$, and two among the three with one $a$ and two $b$'s. They produce neither $a_1a_2a_3$ nor $b_1b_2b_3$. Thus the rank is $12+2+2=16$. The irreducibility argument excludes the remaining $\operatorname{Sym}^3U_E$ from coniveau one, proving \eqref{eq:divisorconiveau} and $\varrho=6$.

For the CM cube the Hodge group is the one-dimensional norm-one torus of the imaginary quadratic endomorphism field. Its degree-two invariant subspace has dimension $9$ and complexifies to all of $H^{1,1}$. The Lefschetz $(1,1)$ theorem identifies it with $\operatorname{NS}(E^3)_{\mathbb{Q}}$. Wedge products with $H^1$ span all of $H^{2,1}\oplus H^{1,2}$: every exterior monomial with two holomorphic letters and one antiholomorphic letter, or conversely, is a divisor product. The dimension is $18$, so the Hodge-type upper bound gives equality with $N^1H^3$.

Finally, the cohomology of an abelian threefold is the exterior algebra on $H^1$, with three dimensions in each of types $(1,0)$ and $(0,1)$. This gives the stated Hodge numbers and Betti number. Its underlying real Lie group is a six-dimensional compact torus.
\end{proof}

For these examples $\Lambda_0$ and $\Lambda_2$ are abstractly $\mathbb{Z}^6$, while $\Lambda_3=0$. The ordered rank families are therefore
\begin{equation}\label{eq:rankfamilies}
 \begin{gathered}
 (6,6,6,0),\qquad(6,12,6,0),\qquad(6,14,6,0),\\
 (6,16,6,0),\qquad(6,18,6,0).
 \end{gathered}
\end{equation}
These distinguish the coniveau-dependent cycle quotients despite identical Hodge numbers and diffeomorphic underlying manifolds. At this dimension, the distinctions are classical coniveau distinctions expressed through the quotient-lattice family, not additional lattices beyond the Walker factor.

\begin{corollary}\label{cor:noncmisogeny}
Let $A=E_1\times E_2\times E_3$, where all three elliptic curves are non-CM, and let $s$ be the number of their isogeny classes. Then
\[
 \operatorname{rank}\Lambda_1(A)=18-2s.
\]
Thus the three possible ranks in this class are $12$, $14$, and $16$.
\end{corollary}
\begin{proof}
For an isogeny $f:A\longrightarrow A'$, rational pullback preserves coniveau and has inverse $f_*/\deg f$. Thus $\dim_{\mathbb{Q}}N^1H^3$ and, by Corollary~\ref{cor:classical}, the quotient rank are isogeny-invariant. According as $s=3,2,1$, the variety is isogenous to one of the three non-CM product cases in Proposition~\ref{prop:abelianranks}. Their ranks give the formula.
\end{proof}

Variation of geometric coniveau rank in abelian threefold families is already exhibited in \cite[Example 6.5]{ACMV17}, with ranks $6$ and $18$ in characteristic zero. Here the additional feature is that four distinct rank strata are simultaneously analytically dense and individually nonconstructible in a single explicit family.

\begin{corollary}\label{cor:variation}
There is a smooth projective family of abelian threefolds with constant Hodge numbers for which each of the rank-$12$, rank-$14$, rank-$16$, and rank-$18$ loci of $s\longmapsto\operatorname{rank}\Lambda_1(X_s)$ is analytically dense and is not Zariski constructible. The rank function is nowhere continuous in the analytic topology. It is neither upper nor lower semicontinuous, in either the analytic or the Zariski topology, and is not a Zariski constructible function.
\end{corollary}
\begin{proof}
Let $B=\mathbb{A}^1_{\mathbb{C}}\setminus\{0,1\}$ and let $\mathcal{E}\to B$ be the smooth projective Legendre family
\[
 y^2z=x(x-z)(x-tz).
\]
Over $B^3$, take the product of its three pullbacks. All fibers have the Hodge numbers in Proposition~\ref{prop:abelianranks}.

Fix a non-CM elliptic curve in this family. Its isogeny class in $B$ is countable and analytically dense. For density, in the upper half-plane the orbit of its period $\tau$ under $\operatorname{GL}_2^+(\mathbb{Q})$ is dense, since this group is dense in $\operatorname{GL}_2^+(\mathbb{R})$ and the latter acts transitively. If
\[
 \tau'=\frac{a\tau+b}{c\tau+d}
\]
for a positive-determinant rational matrix, then $(c\tau+d)(\mathbb{Z}+\mathbb{Z}\tau')$ is commensurable with $\mathbb{Z}+\mathbb{Z}\tau$ after clearing denominators. The elliptic curves are thus isogenous. Passing through the analytic period charts for the Legendre family proves the density assertion. Countability follows from the rational matrices and the finite number of level-two structures. The CM parameters are also countable.

Every nonempty analytic open subset of $B^3$ contains a product of three nonempty open sets. In each coordinate choose a parameter in this fixed dense non-CM isogeny class. The resulting threefold is isogenous to a non-CM cube, so its rank is $16$. The rank is invariant under isogeny here: finite pullback on rational cohomology preserves $N^1$ and is inverted by pushforward divided by the degree. In the same open set, choose non-CM parameters successively outside the countably many forbidden isogeny classes. They are pairwise nonisogenous and give rank $12$. To obtain rank $14$ in the same open set, choose the first two parameters in the fixed non-CM isogeny class and the third outside that class and outside the countable CM locus. Corollary~\ref{cor:noncmisogeny} applies. Finally, fix one CM elliptic curve. The same rational-matrix orbit argument makes its isogeny class dense. Choose all three parameters in this class; their product is isogenous to a CM cube and has rank $18$. Thus all four stated rank loci are analytically dense.

At a rank-$12$ point, arbitrarily near rank-$16$ points contradict upper semicontinuity. At a rank-$16$ point, arbitrarily near rank-$12$ points contradict lower semicontinuity. Zariski semicontinuity would imply analytic semicontinuity, since Zariski-open sets are analytically open.

Finally, on an irreducible algebraic variety a constructible integer-valued function is constant on some nonempty Zariski-open subset: among a finite collection of locally closed strata, one contains a dense open subset. But every nonempty Zariski-open subset of $B^3$ is analytically open and contains both ranks $12$ and $16$. This contradiction proves nonconstructibility of the function. Each of the four stated rank loci is Zariski dense, but contains no nonempty Zariski-open subset because a different rank locus is analytically dense. A dense constructible subset of an irreducible variety contains a nonempty Zariski-open subset; thus none of the four loci is constructible. Finally, every analytic neighborhood contains points of ranks $12$ and $18$. A continuous integer-valued function is locally constant at each point, so the rank function is nowhere continuous.
\end{proof}

\section{Infinite Lawson homology and homotopy of the closure}\label{sec:infinite}
The Lawson existence results were stated in \cite[Theorem 1.7 and Corollary 1.8]{HuAJ}, and the singular rational examples were constructed in \cite{HuIG}. For the smooth examples we supply the passage from Clemens's full Abel--Jacobi invariant to its top form projection. The differential identity uses standard normal-function horizontality \cite[Section 2.1]{MW}; finite algebraic spreads and a Baire argument supply simultaneous projected independence. Products are then controlled by finite dimensional period kernels. The applications below concern the associated closures and Chow quotients.

\subsection{The Hodge range and the projected invariant}
Set $m=k-2p\geq0$, so $k$ is the Lawson homological degree and $m$ is the homotopy degree. We reserve $\Phi_{p,k}$ for the cycle map and write $\operatorname{AJ}^{\mathrm{tr}}_{p,m}$ for the projected generalized Abel--Jacobi homomorphism of \cite[Section 3, equation (3.4)]{HuAJ}. For a subspace $F\subset H^j(X,\mathbb{C})$, denote the image of integration $H_j(X,\mathbb{Z})\to F^*$ by $\operatorname{per}_{F}H_j(X,\mathbb{Z})$. Then
\begin{equation}\label{eq:regulator}
 \operatorname{AJ}^{\mathrm{tr}}_{p,m}:L_pH_{2p+m}(X)_{\mathrm{hom}}\longrightarrow
 \frac{\bigl(F^{p+m+2}H^{2p+m+1}(X,\mathbb{C})\bigr)^*}
 {\operatorname{per}_{F^{p+m+2}}H_{2p+m+1}(X,\mathbb{Z})}.
\end{equation}
The period quotients in this section are abstract groups; their period subgroups need not be discrete and are not replaced by their closures.

For precision, the trace and homotopy trace in \cite[Remark 3.3 and Section 3.2]{HuAJ} have bidimensions $(p+i,p+m-i)$, $0\leq i\leq m$, and $(p+i,p+m+1-i)$, $0\leq i\leq m+1$, respectively. Thus integration on representatives uses $F^{p+m+1}$, whereas homotopy invariance uses $F^{p+m+2}$. Write $\mathcal{A}^j(X)$ for the smooth complex-valued differential forms of degree $j$, with their Hodge filtration. Filtered de Rham strictness on the compact K\"ahler manifold $X$ \cite[Sections 5--6]{DGMS} ensures that an exact closed form in $F^q\mathcal{A}^j(X)$ has a primitive in $F^q\mathcal{A}^{j-1}(X)$. Thus, if $\partial U=c_\varphi$ is the trace of a based cycle map, the functional $\omega\mapsto\int_U\omega$ depends only on its filtered cohomology class: Stokes's theorem reduces a change of representative to integration on $c_\varphi$, which vanishes by its stated bidimensions. Changing $U$ adds an integral period, not necessarily the number zero. Under a based homotopy, the difference of the bounding chains differs from the homotopy trace by an integral cycle. The homotopy trace annihilates the filtration in \eqref{eq:regulator}. In the offset notation $H^{p+a,p+b}$ of \cite{HuAJ}, one must therefore use $a+b=m+1$ with $a\geq m+1$ before descent and $a>m+1$ afterwards; $a,b$ need not both be nonnegative. In particular, the relevant threefold projection is $F^3H^3=H^{3,0}$, not the full classical space $F^2H^3=H^{3,0}\oplus H^{2,1}$.

For a smooth projective threefold $Y$, set
\begin{equation}\label{eq:topAJ}
 V_Y=H^{3,0}(Y)^*,\qquad P_Y=\operatorname{per}_{H^{3,0}}H_3(Y,\mathbb{Z}),\qquad
 \operatorname{AJ}_Y^{\mathrm{top}}:\operatorname{Griff}_1(Y)\longrightarrow V_Y/P_Y.
\end{equation}
Here $\operatorname{Griff}_1(Y)=Z_1(Y)_{\mathrm{hom}}/G_1(Y)$, and $\operatorname{AJ}_Y^{\mathrm{top}}$ is the restriction of the classical Abel--Jacobi functional to $H^{3,0}(Y)$. It vanishes on algebraically trivial cycles by the path-trace calculation just recalled. For $\partial U=W$, its value is represented by $b_W(\beta)=\int_U\beta$. The map \eqref{eq:topAJ} agrees with \eqref{eq:regulator} when $(p,m)=(1,0)$.

\subsection{The passage from Clemens's theorem}
Let $S$ be the open parameter space of smooth quintic hypersurfaces in $\mathbb{P}^4$, write $S^{\mathrm{an}}$ for its complex-analytic space, and let $f:\mathcal{Y}\to S$ be the universal family. Write $\mathcal{H}_{\mathbb{Z}}=R^3f_*\mathbb{Z}$, $\mathcal{H}_{\mathbb{Q}}=\mathcal{H}_{\mathbb{Z}}\otimes\mathbb{Q}$, and $\mathcal{H}=\mathcal{H}_{\mathbb{Z}}\otimes\mathcal{O}_S$, with Gauss--Manin connection $\nabla$, Hodge filtration $F^\bullet$, and cup-product pairing $\langle\ ,\ \rangle$. The corresponding intermediate Jacobian family is $\mathcal{H}/(F^2\mathcal{H}+\mathcal{H}_{\mathbb{Z}})$. The local holomorphic lift $v$ of a cycle normal function satisfies $\nabla v\in F^1\mathcal{H}\otimes\Omega_S^1$ \cite[Section 2.1, equations (2.3) and (2.9)]{MW}. For a local generator $\Omega$ of $F^3$,
\begin{equation}\label{eq:completevariation}
 F_s^2\mathcal{H}=\operatorname{span}_{\mathbb{C}}\{\Omega_s,(\nabla_\xi\Omega)_s:\xi\in T_sS\}.
\end{equation}
Indeed, contraction with $\Omega$ identifies $H^1(T_{Y_s})$ with $H^{2,1}(Y_s)$, and the derivative of the period map is the Kodaira--Spencer map followed by contraction \cite[Section 3(b), Proposition (3.6)]{GriffithsPeriods}. The embedded quintic deformations are complete: the restricted Euler sequence gives $H^1(T_{\mathbb{P}^4}|_{Y_s})=0$, and the normal sequence gives the required surjectivity. These statements remain valid after an \'{e}tale base change.

\begin{lemma}\label{lem:firstjet}
Let $v$ be a local holomorphic lift of a normal function on a connected base satisfying \eqref{eq:completevariation}, and let $\delta$ be a flat rational section. If $g=\langle v-\delta,\Omega\rangle$ vanishes identically, then the normal function is torsion on that connected neighborhood. More precisely, $g(s)=0$ and $dg(s)=0$ imply $v(s)-\delta(s)\in F_s^2\mathcal{H}$.
\end{lemma}
\begin{proof}
Horizontality and $\langle F^1,F^3\rangle=0$ give
\[
 \partial_\xi g=\langle v-\delta,\nabla_\xi\Omega\rangle.
\]
Vanishing of the value and all first derivatives therefore makes $v-\delta$ orthogonal to $F^2$ by \eqref{eq:completevariation}. Since $(F^2)^\perp=F^2$, the asserted membership follows. If $g$ vanishes identically, clear the denominators of the fixed rational flat section $\delta$; a nonzero integer then annihilates the normal function.
\end{proof}

\begin{theorem}\label{thm:projectedquintic}
There is a dense $G_\delta$ subset $S^\dagger\subset S^{\mathrm{an}}$ such that
\begin{equation}\label{eq:projectedrank}
 \dim_{\mathbb{Q}}\bigl(\operatorname{im}\operatorname{AJ}_{Y_s}^{\mathrm{top}}\otimes\mathbb{Q}\bigr)=\infty
 \qquad(s\in S^\dagger).
\end{equation}
The genericity asserted here is analytic, not a claim that the exceptional loci are algebraic.
\end{theorem}
\begin{proof}
We use the geometric form of Clemens's theorem \cite[Theorem (0.2), Theorem (1.27), and Section 3, pp.~36--37]{Clemens}. For each positive integer $t$, it supplies $t$ independent full Abel--Jacobi values arising from cycles
\[
 W_i=5C_i-(\deg C_i)D,\qquad D=Y\cap\mathbb{P}^2,\qquad
 N_{C_i/Y}\cong\mathcal{O}_{\mathbb{P}^1}(-1)^{\oplus2}.
\]
These cycles are homologically trivial because $\deg D=5$ and $H_2(Y,\mathbb{Z})\cong\mathbb{Z}$. Since $H^0(N_{C_i/Y})=H^1(N_{C_i/Y})=0$, the relative Hilbert scheme is \'{e}tale at each chosen curve. Take the component through the chosen tuple in a finite fiber product of the relative Hilbert schemes, then its closure and normalization. After restricting to a dense Zariski open $S_t\subset S$, this gives a connected finite \'{e}tale surjection $h_t:T_t\to S_t$ carrying the $t$ cycle families. The restriction may also be chosen so that these curves are smooth and the plane-section family is flat.

The resulting normal functions $\nu_1,\ldots,\nu_t$ have no nonzero integer combination which is a torsion section. To see this even if the original point was removed on shrinking, work on the smooth open locus of the same irreducible spread where the cycle families are defined. This is connected and contains both the original tuple and a dense open subset of $T_t$. A torsion relation, after multiplication by its annihilating integer, extends by the identity theorem and would contradict the independence at the original tuple.

Cover $T_t$ by countably many simply connected coordinate neighborhoods where $h_t$ is biholomorphic onto its image and the normal functions lift to sections $v_i$. A rational dependence between their projected values is equivalent to
\begin{equation}\label{eq:dependencezero}
 \left\langle\sum_{i=1}^t c_iv_i-\delta,\Omega\right\rangle=0,
 \qquad 0\neq(c_1,\ldots,c_t)\in\mathbb{Z}^t,
 \qquad\delta\in\Gamma(\mathcal{H}_{\mathbb{Q}})\ \text{flat}.
\end{equation}
There are countably many such equations on each neighborhood. None vanishes identically, by Lemma~\ref{lem:firstjet} and the absence of torsion relations. Their zero sets are therefore proper analytic subsets.

Exhaust the neighborhoods by compact subsets. The image under $h_t$ of the intersection of any such compact subset with any zero set is closed and nowhere dense in $S^{\mathrm{an}}$: locally $h_t$ is biholomorphic. Include also $S\setminus S_t$. The excluded images form a countable union of closed nowhere-dense subsets. At a point outside this union, a lift to $T_t$ supplies $t$ independent projected values on the fiber. Take the union of the excluded sets for all $t\geq1$ and apply the Baire theorem. Its complement is a dense $G_\delta$ on which the projected rank is at least $t$ for every $t$, proving \eqref{eq:projectedrank}.
\end{proof}
Only a finite collection is spread at each stage of this proof. No common neighborhood for infinitely many curve families is required. Nor is injectivity of the projection from the full intermediate Jacobian on a fixed fiber asserted.

\subsection{Products and finite dimensional period kernels}
\begin{proposition}\label{prop:ellipticrepair}
Let $Y$ be a smooth projective threefold satisfying \eqref{eq:projectedrank}. For every elliptic curve $E$ and every nonzero class $[\gamma]\in H_1(E,\mathbb{Z})$, choose a loop $\gamma:S^1\to E$ based at the origin. The values of $\operatorname{AJ}^{\mathrm{tr}}_{1,1}$ on the based cycle loops
\[
 \varphi_{\gamma,W}(t)=([\gamma(t)]-[0])\times W,
 \qquad W\in Z_1(Y)_{\mathrm{hom}},
\]
span an infinite dimensional rational vector space after tensoring the target with $\mathbb{Q}$.
\end{proposition}
\begin{proof}
Write $E=\mathbb{C}/(\mathbb{Z}+\tau\mathbb{Z})$ with $\operatorname{Im}\tau>0$ and $[\gamma]=aA+bB$, where the $dz$-periods of $A,B$ are $1,\tau$. Set $V=V_Y$, $P=P_Y$, and $P_{\mathbb{Q}}=\operatorname{span}_{\mathbb{Q}}P$. The integral K\"unneth decomposition, as in \cite[pp.~1253--1254]{HuAJ}, identifies the projected product-period subgroup with $P+\tau P\subset V$. The trace is $\gamma\times W$, bounded by $-\gamma\times U$ when $\partial U=W$. Hence
\begin{equation}\label{eq:ellipticvalue}
 \operatorname{AJ}^{\mathrm{tr}}_{1,1}([\varphi_{\gamma,W}])
 =-(a+b\tau)b_W+(P+\tau P).
\end{equation}
In particular the Lawson class is homologically trivial. Set $c=a+b\tau\neq0$ and $M_\tau=P_{\mathbb{Q}}+\tau P_{\mathbb{Q}}$. Multiplication by $c$ induces
\[
 T_c:V/P_{\mathbb{Q}}\longrightarrow V/M_\tau,
 \qquad \ker T_c\cong M_\tau/cP_{\mathbb{Q}},
 \qquad \dim_{\mathbb{Q}}\ker T_c\leq\dim_{\mathbb{Q}}P_{\mathbb{Q}}<\infty.
\]
Here $cP_{\mathbb{Q}}\subset M_\tau$, and multiplication by $c$ is a rational linear automorphism of $V$. Since $(V/P)\otimes\mathbb{Q}=V/P_{\mathbb{Q}}$, the projected infinite dimensional image remains infinite dimensional under \eqref{eq:ellipticvalue}.
\end{proof}
This proves the conclusion of \cite[Proposition 4.3]{HuAJ} under its stated projected hypothesis and verifies that hypothesis for the quintics of Theorem~\ref{thm:projectedquintic}. The elliptic curve is arbitrary. No assertion that every individual non-torsion class survives is needed.

\begin{proposition}\label{prop:zerocycleproduct}
Suppose $Y$ satisfies \eqref{eq:projectedrank}, and let $M$ be a smooth connected projective $m$-fold with $H^{m,0}(M)\neq0$. For $m=0$ take $M$ to be a point. Then
\[
 \dim_{\mathbb{Q}}\bigl(\operatorname{im}\operatorname{AJ}^{\mathrm{tr}}_{1,m}(M\times Y)\otimes\mathbb{Q}\bigr)=\infty.
\]
\end{proposition}
\begin{proof}
For $m\geq1$, choose $\xi\in H_m(M,\mathbb{Z})$ whose period functional $a_\xi\in H^{m,0}(M)^*$ is nonzero. Dold--Thom \cite{DoldThom} represents it by a based map $u:S^m\to Z_0(M)^0$, where the superscript $0$ denotes degree zero cycles. Exterior product gives $\varphi_{u,W}(t)=u(t)\times W$. The graphing construction is compatible with this product: its trace is $c_u\times W$, bounded by $(-1)^mc_u\times U$ for $\partial U=W$. Consequently
\begin{equation}\label{eq:zeroproductvalue}
 \operatorname{AJ}^{\mathrm{tr}}_{1,m}([\varphi_{u,W}])(\eta\wedge\beta)
 =(-1)^m a_\xi(\eta)b_W(\beta)\quad\text{modulo periods}.
\end{equation}
Here $\eta\in H^{m,0}(M)$ and $\beta\in H^{3,0}(Y)$. The form has type $(m+3,0)$, in the precise range of \eqref{eq:regulator}. Let $R_{\mathbb{Q}}\subset H^{m,0}(M)^*\otimes_{\mathbb{C}}V_Y$ be the rational span of the projected integral periods of $H_{m+3}(M\times Y,\mathbb{Z})$. This is finite dimensional over $\mathbb{Q}$ and contains $a_\xi\otimes (P_Y)_{\mathbb{Q}}$. Since $j_\xi:v\mapsto a_\xi\otimes v$ is injective, its induced map modulo periods has finite dimensional kernel
\[
 j_\xi^{-1}(R_{\mathbb{Q}})/(P_Y)_{\mathbb{Q}}.
\]
Equation~\eqref{eq:zeroproductvalue} proves the assertion. The case $m=0$ is \eqref{eq:projectedrank}.
\end{proof}
The sphere here maps into $Z_0(M)$, not necessarily into $M$. In particular, $M=E^m$ is permitted even though $\pi_m(E^m)=0$ for $m\geq2$. No spherical-period assertion for the underlying variety is used.

\begin{proof}[Proof of Theorem~\ref{thm:infiniteexistence}]
Choose one $Y$ from Theorem~\ref{thm:projectedquintic}. For $p\geq1$ and $m\geq0$, take $M=E^m$ in Proposition~\ref{prop:zerocycleproduct} and multiply its cycle families by $[\mathbb{P}^{p-1}]$. If $h$ is the hyperplane class, normalized by $\int_{\mathbb{P}^{p-1}}h^{p-1}=1$, the test form $\eta\wedge\beta\wedge h^{p-1}$ has type
\begin{equation}\label{eq:finaltype}
 (m,0)+(3,0)+(p-1,p-1)=(p+m+2,p-1).
\end{equation}
Its degree is $2p+m+1$ and its first Hodge index equals the cutoff in \eqref{eq:regulator}. Integral K\"unneth shows that the restricted period subgroup is unchanged by the added projective factor. The infinite dimensional regulator image therefore persists. This gives infinite rational rank in $L_pH_{2p+m}(X_{p,m})_{\mathrm{hom}}$, and hence infinite generation of $L_pH_{2p+m}(X_{p,m})$. The dimension is $p+m+2$. For $m=0$, one may alternatively apply the projective bundle theorem directly to Clemens's Griffiths group.

For $m\geq2$, Theorem~\ref{thm:main} identifies the Lawson group with $\pi_m\operatorname{CR}_p(X_{p,m})$. The split Hurewicz injection of Corollary~\ref{cor:hodgeclosure} then gives infinite rational rank in $H_m(\operatorname{CR}_p(X_{p,m}),\mathbb{Z})$ as well.
\end{proof}

\begin{corollary}\label{cor:simultaneous}
Fix $p\geq1$, $a\geq0$, an elliptic curve $E$, and the quintic $Y$ of Theorem~\ref{thm:infiniteexistence}. On the single smooth projective variety
\[
 X_{p,a}=E^a\times Y\times\mathbb{P}^{p-1}
\]
one has
\[
 \dim_{\mathbb{Q}}\bigl(L_pH_{2p+m}(X_{p,a})_{\mathrm{hom}}\otimes\mathbb{Q}\bigr)=\infty
 \qquad(0\leq m\leq a).
\]
For $2\leq m\leq a$, both $\pi_m\operatorname{CR}_p(X_{p,a})$ and $H_m(\operatorname{CR}_p(X_{p,a}),\mathbb{Q})$ have infinite rational rank. In particular, if $a\geq2$, this closure is not finitely dominated and has no finite CW homotopy type.
\end{corollary}
\begin{proof}
For $m\leq a$, insertion of the origins of the last $a-m$ elliptic factors and projection give proper morphisms
\[
 X_{p,m}\xrightarrow{i}X_{p,a}\xrightarrow{\operatorname{pr}}X_{p,m},
 \qquad\operatorname{pr}\circ i=\operatorname{id}.
\]
Their pushforwards are retractions on cycle spaces, Lawson homology, and rational equivalence closures, by Proposition~\ref{prop:functoriality}. They preserve the homologically trivial subgroups. Thus the indicated groups for $X_{p,m}$ are direct summands of those for $X_{p,a}$. Theorem~\ref{thm:infiniteexistence} proves the assertions. A finitely dominated space, meaning a homotopy retract of a finite CW complex, has finite dimensional rational homology in every degree. The infinite rank homology just obtained therefore excludes finite domination, and in particular finite CW homotopy type.
\end{proof}

\begin{proposition}\label{prop:infiniteone}
For a connected projective $X$, possibly singular, if $L_pH_{2p+1}(X)\otimes\mathbb{Q}$ is infinite dimensional, then at least one of $\pi_1\operatorname{CR}_p(X)$ and $\Lambda_p(X)$ has infinite rational rank. Finite rank of either group forces infinite rank of the other. In particular this alternative occurs on $X_{p,a}$ for $a\geq1$.
\end{proposition}
\begin{proof}
Tensor \eqref{eq:exact} with $\mathbb{Q}$ and apply Corollary~\ref{cor:simultaneous}. An infinite dimensional middle term cannot have two finite dimensional outer terms.
\end{proof}
The generalized Abel--Jacobi target in \eqref{eq:regulator} is not asserted to be Hausdorff and is not the classical cycle level Jacobian. The construction does not provide a factorization through $\Lambda_p(X)$. Thus infinite Lawson rank does not, by itself, realize an infinite quotient lattice. Also, the degree $k=2p$ concerns $\pi_0Z_p(X)$, not $\pi_0\operatorname{CR}_p(X)$; the latter is always zero.

\subsection{Cycle level Walker kernels}
\begin{corollary}\label{cor:infinitewalker}
For the varieties $X_{p,a}$ of Corollary~\ref{cor:simultaneous}, with $a\geq1$, the cycle level Walker kernel satisfies
\[
 \dim_{\mathbb{Q}}\bigl(\pi_mW_p(X_{p,a})\otimes\mathbb{Q}\bigr)
 =\dim_{\mathbb{Q}}H_m(W_p(X_{p,a}),\mathbb{Q})=\infty
 \qquad(1\leq m\leq a).
\]
\end{corollary}
\begin{proof}
In degree one, \eqref{eq:walkerkernelhomotopy} identifies $\pi_1W_p$ with $\ker\overline{\Phi}_{p,2p+1}$, which contains the infinite rank subgroup $L_pH_{2p+1}(X_{p,a})_{\mathrm{hom}}$. For $m\geq2$, the same equation identifies $\pi_mW_p$ with the corresponding Lawson group, and Corollary~\ref{cor:simultaneous} applies. The group $W_p$ is closed in $G_p(X)$ and inherits a compact final-topology filtration. The proof of \cite[Lemma 3.5.9, pp.~102--103]{LF04} therefore gives a left inverse to its Hurewicz map, proving the homology assertion. This does not identify $W_p$ with $\operatorname{CR}_p$ in degree one without the independent hypotheses in Section~\ref{sec:invariants}.
\end{proof}

\begin{remark}\label{rem:remaining}
The integral quotient in Theorem~\ref{thm:walkerquotient} does not determine its residual lattice $K_p^{\mathrm{per}}(X)$ for a general smooth projective variety. No fixed smooth projective example with an infinite rank $\Lambda_p(X)$ is constructed here. Infinite Lawson rank in degree $2p+1$ gives only the alternative in Proposition~\ref{prop:infiniteone}.

The abstract group $D_p(X)=\operatorname{CR}_p(X)/R_p(X)$ and the homotopy groups of the unclosed subgroup $R_p(X)$ are different questions. Neither density nor the global product decomposition identifies $R_p(X)$ with its closure up to weak equivalence. The normal-function argument establishes analytic residual genericity, not algebraicity of the additional exceptional loci. These limitations do not contradict the structure theorem: they distinguish the localization of higher homotopy in the closure from a computation of the Lawson groups themselves. In particular, the local contractibility established here does not imply finite CW homotopy type.
\end{remark}

\end{document}